\documentclass[11pt]{article} 

\usepackage[utf8]{inputenc} 
\usepackage[british]{babel} 

\begin{hyphenrules}{british}
\end{hyphenrules}

\usepackage{geometry} 
\usepackage{graphicx} 

\usepackage[parfill]{parskip} 

\usepackage{hyperref} 
\usepackage{booktabs} 
\usepackage{array} 
\usepackage{paralist} 
\usepackage{verbatim} 
\usepackage{subfig} 
\usepackage{bmpsize}
\usepackage{amssymb}
\usepackage{mathtools}
\usepackage{amsmath}	
\usepackage{amsthm}
\usepackage{graphicx}
\usepackage{faktor} 
\usepackage{mathrsfs} 
\usepackage[all, cmtip]{xy} 
\usepackage{multicol} 
\usepackage{diagbox} 
\usepackage{multirow} 
\usepackage{tikz} 
\usepackage{tikz-cd} 
\usetikzlibrary{decorations.pathreplacing}
\usepackage{array, makecell, rotating}
\usepackage{longtable} 
\usepackage{adjustbox} 
\usepackage{arydshln} 
\usepackage{halloweenmath} 
\usepackage{array, makecell, rotating}
\usepackage{longtable} 
\setcellgapes{2pt}
\newcommand\nocell[1]{\multicolumn{#1}{c|}{}}

\makeatletter
\renewcommand*\env@matrix[1][\arraystretch]{%
  \edef\arraystretch{#1}%
  \hskip -\arraycolsep
  \let\@ifnextchar\new@ifnextchar
  \array{*\c@MaxMatrixCols c}}
\makeatother  

\numberwithin{equation}{subsection}

\newcommand\Tstrut{\rule{0pt}{3.5ex}}       
\newcommand\Bstrut{\rule[-1.4ex]{0pt}{0pt}} 
\newcommand\TBstrut{\Tstrut\Bstrut}         
\newcommand\TTstrut{\rule{0pt}{6ex}}       

\allowdisplaybreaks

\newtheorem{theorem}{Theorem}[subsection]

\theoremstyle{remark}
\newtheorem{definition}[theorem]{Definition}
\theoremstyle{theorem}
\newtheorem{proposition}[theorem]{Proposition}
\newtheorem{lemma}[theorem]{Lemma}
\newtheorem{corollary}[theorem]{Corollary}
\theoremstyle{remark}
\newtheorem{remark}[theorem]{Remark}
\theoremstyle{remark}

\usepackage{fancyhdr} 
\usepackage{sectsty}
\allsectionsfont{\sffamily\mdseries\upshape} 

\usepackage[nottoc,notlof,notlot]{tocbibind} 
\usepackage[titles,subfigure]{tocloft} 

\title{K3 surfaces with a symplectic action of $(\mathbb Z/2\mathbb Z)^2$}
\author{Benedetta Piroddi}
\date{} 

\begin{document}
\maketitle
\begin{abstract}
We study the symplectic action of the group $(\mathbb Z/2\mathbb Z)^2$ on a K3 surface $X$: we describe its action on $H^2(X,\mathbb Z)$ and the maps induced in cohomology by the rational quotient maps; we give a lattice-theoretic characterization of $\tilde Z$, the resolution of singularities of the quotient $X/\iota$, where $\iota$ is any of the involutions in $(\mathbb Z/2\mathbb Z)^2$. Assuming $X$ is projective, we describe the correspondence between irreducible components of its moduli space, and those of $\tilde Z$ and $\tilde Y$ (the resolution of singularities of $X/(\mathbb Z/2\mathbb Z)^2$): this being the first description of this correspondence for a non-cyclic action, we see new phenomena, of which we provide explicit examples assuming $X$ has a polarization of degree 4.
\end{abstract}

\section*{Introduction}
If $X$ is a K3 surface, and $G$ is a finite group acting symplectically on $X$, the quotient $Y=X/G$ has only $ADE$ singularities, and therefore admits another K3 surface $\tilde Y$ as resolution. As it was already established by Nikulin in \cite{Nikulin2}, the moduli spaces of $X$ and $\tilde Y$ are in bijection. If $G$ is not simple, one can consider also the intermediate quotient surfaces $\tilde Z_H$, obtained as resolution of quotients of the form $X/H$, where $H$ is a normal subgroup of $G$: the moduli space of $\tilde Z_H$ is again in bijection with that of $X$ and $\tilde Y$.\\
In the projective case, the moduli space of $X$ splits in countable irreducible components -- here called \emph{projective families} -- at least one for every choice of the degree $2d=L^2$ of the polarization $L$ of $X$; if there is more than one family, the corresponding quotient surfaces may be naturally polarized with different degrees. A complete description of the correspondence between projective families of $X$ and its quotients is only known for $G=\mathbb Z/n\mathbb Z$, $n=2,3,4$; symplectic involutions are studied in \cite{VGS} and \cite{GS}, automorphisms of order 3 and 4 in \cite{GP} and \cite{P} respectively.

In this paper, we study the symplectic action of $G=(\mathbb Z/2\mathbb Z)^2$ on a K3 surface. Following the same outline as \cite{GP}, \cite{P}, we start by selecting a K3 surface $X$ with high Picard rank and a Jacobian fibration $\pi: X\rightarrow\mathbb P^1$ such that $MW(\pi)\simeq G$: the resulting description of the action of $G$ on $H^2(X,\mathbb Z)$ holds for any K3 surface thanks to \cite[Thm. 4.7]{Nikulin2}. We then study the maps induced in cohomology by the rational quotient maps, and we give a lattice-theoretic characterization of the intermediate quotien surface $\tilde Z$ (see Thm. \ref{Ztaunonproj}).  We follow the same approach as \cite{P} to describe the correspondence between projective families of $X,\tilde Z,\tilde Y$: we fix an embedding $\Omega_{2,2}\hookrightarrow\Lambda_{\mathrm{K3}}$ of the co-invariant lattice for the action of $G$, and choose for each projective family of $X$ a representative class $L\in\Omega_{2,2}^{\perp_{\Lambda_{\mathrm{K3}}}}$ that takes the role of $\Omega_{2,2}^{\perp_{NS(X)}}$; the correspondence is given by taking the image of $L$ through the quotient maps.\\
We remark that the image of the map $\pi_{2,2*}$, induced in cohomology by the rational quotient map $\pi_{2,2}: X\dashrightarrow \tilde Y$, is not primitive in $H^2(\tilde Y,\mathbb Z)$, being instead a sublattice of index 2 (see Remark \ref{gammamezzi}). A similar result holds for its dual $\pi_{2,2}^*$, its image being a sublattice of index $2^3$ in $H^2(X,\mathbb Z)$  (see Corollary \ref{pi22notprimitive}).

The correspondence between projective families of $X$ and its quotients is quite intricate: as already anticipated by Nikulin \cite{Nikulin2}, projective families of $X$ and $\tilde Y$ are in bijection, but their relation to projective families of $\tilde Z$ presents a wide range of phenomena (see Theorem \ref{correspondence:GammaKlein}).
As for $\mathbb Z/4\mathbb Z$ we see that two family of $X$ can collide on the same family of $\tilde Z$: this happens for two families in every degree $2(8h+4)$. 
If $NS(X)$ is an overlattice of index 2 of $\Omega_{2,2}\oplus\langle 2(4h+2)\rangle$, then $X$ admits two different actions of $(\mathbb Z/2\mathbb Z)^2$; since they do not generally preserve the same ample class on $X$, however, there is no projective model that displays both as actions induced by the ambient space. For one of these actions, the projective family the intermediate quotient surface $\tilde Z$ belongs to is determined; for the other action, the three involutions of $(\mathbb Z/2\mathbb Z)^2$ act differently on the polarization: more precisely, there are always two (call them $\tau, \varphi$) such that the intermediate quotient surfaces $\tilde Z_\tau, \tilde Z_\varphi$ belong to different projective families. This latter asymmetry appears also for a projective family in degree $4h$.

Lastly, if $L^2=4h$, and $NS(X)=\Omega_{2,2}\oplus\mathbb ZL$, the action of $(\mathbb Z/2\mathbb Z)^2$ on $\mathbb P(H^0(X,L)^*)$ is induced by an action of the dihedral group of order 8 on $H^0(X,L)$ (see Thm. \ref{prop:autospaziKlein}). This implies that, while one of the generators of $(\mathbb Z/2\mathbb Z)^2$ splits $\mathbb P(H^0(X,L)^*)$ in two eigenspaces, so that we can get a model of the intermediate quotient surface by projection on one of them, the other generator swaps the eigenspaces of the first, so we cannot do the same for the total quotient.

\section*{Aknowledgements}
This paper presents part of the results of the author's PhD thesis. She thanks the PhD program of Università degli Studi di Milano for the support, and is especially grateful to her tutor Alice Garbagnati. The author is currently supported by the Prin project 2022 \emph{Moduli spaces and special varieties}, and she's a member of INdAM GNSAGA.

\section{The symplectic action of  $(\mathbb Z/2\mathbb Z)^2$ on the K3 lattice}

As for any K3 surface $X$ the symplectic action of $(\mathbb Z/2\mathbb Z)^2$ on $X$ induces always the same action as isometry group on the \emph{K3 lattice} $\Lambda_{\mathrm{K3}}\simeq H^2(X,\mathbb Z)$ \cite[Thm. 4.7]{Nikulin2}, we can describe this action starting from a K3 surface of our choice. We are going to use a K3 surface with a Jacobian fibration $X\xrightarrow{p}\mathbb P^1$ such that $MW(p)=(\mathbb Z/2\mathbb Z)^2$: indeed the Mordell-Weil group $MW(p)$ acts on $X$ by translation on each fiber of $p$, therefore providing a symplectic action on $X$. Moreover, from the relation \cite[Thm. 6.3]{SchuttShioda}
\begin{equation}\label{eqn:MW}
MW(p)\simeq NS(X)/{\mathcal{T}(p)}
\end{equation}
we deduce that, if $MW(p)$ is finite, $NS(X)$ is an overlattice of finite index of the \emph{trivial lattice} ${\mathcal{T}(p)}$: this is the lattice generated by the generic fiber $F$ of $p$, the image of the zero section $s=s(\mathbb P^1)$ and the irreducible components of the reducible fibers which do not intersect the curve $s$. We compute the action of $MW(p)$ on $NS(X)$ using the description of the latter as overlattice of ${\mathcal{T}(p)}$; since this action is symplectic, and therefore trivial on the transcendental lattice $T_X$, it extends uniquely to $H^2(X,\mathbb Z)$.

\subsection{The surface $X_{\omega}$}\label{sec:surfaceKlein}

Let $\omega=e^{i\pi/3}$, consider the elliptic curve $E_\omega=\mathbb C/(\mathbb Z\oplus\mathbb Z\omega)$ and define the K3 surface $X_{\omega}=Kum(E_{\omega}\times E_{\omega})$: it is the unique K3 surface whose transcendental lattice is \[T(X_{\omega})=\begin{bmatrix}4 & 2 \\ 2 & 4 \end{bmatrix}.\]
A description of all the possible Jacobian fibrations on $X_{\omega}$ is provided by Nishiyama \cite[Table 1.3]{Nishiyama}: in particular, there is a fibration 
$\pi: X_{\omega}\rightarrow\mathbb{P}^1$ such that $MW(\pi)\simeq(\mathbb{Z}/{2\mathbb{Z}})^2$. The reducible fibers of $\pi$ are one of type $I^*_{6}$, one of type $I_{6}$ and three of type $I_{2}$ that we'll denote $I^j_{2},\ j=1,2,3$. Call $C_0$ (respectively $D_0,\ E^j_0$) the component of $I^*_{6}$ (resp. $I_{6},\ I^j_2$) intersected by the curve $s$, and number the other components so that, for every $k\in\mathbb{Z}/6\mathbb{Z}$, $D_i$ intersects only $D_{(k+1)}$ and $D_{(k-1)}$; $C_0,C_1$ intersect only $C_2$; $C_9,C_{10}$ intersect only $C_8$ and, for $i=2,\dots,8$, $C_i$ intersects both $C_{(i+1)}$ and $C_{(i-1)}$; moreover, it holds $E^j_0E^j_1=2,\ j=1,2,3$.\\
Using the height pairing formula \cite[\S 11.8]{SchuttShioda} we can determine the components of the reducible fibers $C_i, D_k, E^j_m$ that have non-trivial intersection with the elements of $MW(\pi)$: the zero section $s$ intersects the components $C_0, D_0$ and $E^j_0$; we call $t$ the section that intersects the components $C_1, D_3$ and $E^j_1$, $r$ the section that intersects the components $C_{10},D_0$ and $E^j_1$, and $q=t+r$ (where + is the sum in $MW(\pi)$): therefore $q$ intersects the components $C_9,D_3$ and $E^j_0$. Since $\mathcal T(\pi)$ generates $NS(X_\omega)$ over $\mathbb Q$ by \eqref{eqn:MW} ($MW(\pi)$ is finite), we can write $t,r,q$ uniquely as
\begin{align*}
t&=2F+s-(\sum_{i=1}^8 C_i+D_2+D_3+D_4)-(C_9+C_{10}+D_1+D_3+D_5+\sum_{j=1}^3 E^j_1)/2;\\
r&=2F+s-(\sum_{i=1}^8 iC_i+4C_9+5C_{10}+\sum_{j=1}^3 E^j_1)/2;\\
q&=2F+s-(\sum_{i=1}^8 iC_i+5C_9+4C_{10}+D_1+2D_2+3D_3+2D_4+D_5)/2;
\end{align*}
a $\mathbb{Z}$-basis of $NS(X_{\omega})$ is $\mathcal{B}=\{F, s, t, r, C_2,\dots , C_{10}, D_1,\dots, D_5,E^1_1,E^2_1\}$.

\subsection{The action of $(\mathbb Z/2\mathbb Z)^2$ on the K3 lattice}\label{sec:azioni}

In the previous section we have described $NS(X_\omega)$ as overlattice of finite index of $\mathcal T(\pi)$; however, to compute in Section \ref{sec:quotients_Klein} the maps induced in cohomology by the quotient maps, we will use another description.  We find a sublattice $W$ of finite index of $H^2(X_\omega,\mathbb Z)$ such that the symplectic involutions $\tau,\rho$ and $\varphi=\rho\circ\tau$ (corresponding respectively to the translation by the sections $t,r,q\in MW(\pi)$) act as permutation of orthogonal components of $W$.

There are 8 orthogonal copies of $A_2$ in $NS(X_{\omega})$ that are either fixed or exchanged in pairs by $\tau^*,\rho^*$: $(s,C_0), (t,C_1), (r,C_{10}),(q,C_9),(C_3,C_4),(C_7,C_6),(D_1,D_2),(D_4,D_5)$. Now define
\begin{align*}
S_1&=C_3+2C_4+3C_5+2C_6+C_7,\\
S_2&=4F+2t+2s-(\sum_{i=2}^8iC_i+4C_9+4C_{10}),\\
S_3&=E^1_1-E^2_1,\\
S_4&=-3(E^1_1+E^2_1)-4(2r-2F-t-s+\sum_{i=2}^8(i-1)C_i)+\\
&\ +2(-7C_9-9C_{10}+D_1+2D_2+3D_3+2D_4+D_5):
\end{align*}
then the orthogonal complement in $NS(X_{\omega})$ of the direct sum $A_2^{\oplus 8}$ is isomorphic over $\mathbb Q$ to the lattice $U(3)\oplus A_2(2)=\langle (S_1+S_2)/2,(S_2-S_1)/2\rangle\oplus\langle (S_3+S_4)/2,(S_3-S_4)/2\rangle$; $\tau^*, \rho^*$ act as $id$ on $U(3)$, as $-id$ on $A_2(2)$. We remark that the generators of $U(3)\oplus A_2(2)$ are integral elements in $NS(X_{\omega})$, as it can be seen writing them in the basis $\mathcal B$.\\
We can then write $NS(X_\omega)$ as overlattice of finite index of $A_2^{\oplus 8}\oplus U(3)\oplus A_2(2)$, adding to the list of generators the following (integral) elements, on which we can compute by linearity the action of $\tau^*,\rho^*, \varphi^*$:
\begin{align}\label{generatoriNS}
u_1&=(q-s+C_0-C_3+C_4-C_6+C_7-C_9-D_1+D_2+D_4-D_5)/3,\nonumber\\
u_2&=(r-t+C_1-C_3+C_4-C_6+C_7-C_{10}+D_1-D_2-D_4+D_5)/3,\\
u_3&=(S_1-C_3+C_4+C_6-C_7)/3,\nonumber\\
u_4&=((S_1+S_2)/2-q-r-C_3+C_4+C_9+C_{10})/3,\nonumber\\
u_5&=((S_1+S_2+S_3+S_4)/2+r-C_3+C_4-C_{10}-D_1+D_2+D_4-D_5+S_3)/3.\nonumber
\end{align}

Thanks to \cite[Thm. 4.7]{Nikulin2}, we can now generalize our description of the action of $(\mathbb Z/2\mathbb Z)^2$ on $X_\omega$ to any K3 surface. We will use $\tau$ and $\varphi$ as generators of $(\mathbb Z/2\mathbb Z)^2$: the motivation for this choice is explained in Remark \ref{rem:different_action}.

\begin{proposition}\label{sublattice}
\begin{enumerate}
\item
The isometries $\tau^*$ and $\varphi^*$ act on the sublattice of $\Lambda_{\mathrm{K3}}$ of finite index $W\coloneqq A_2^{\oplus 8}\oplus A_2(2)\oplus{\begin{bmatrix}0 & 6 \\6 & 0\end{bmatrix}}\oplus{\begin{bmatrix}4 & 2 \\2 & 4\end{bmatrix}}$ as follows:
\[
\begin{tikzcd}[cramped, sep=small] 
A_2 \arrow[r,phantom, "\oplus"] \arrow[r, bend left=50, above, "\tau^*", <->] \arrow[rrr, bend right=30, "\varphi^*", swap, <->] &A_2 \arrow[r,phantom, "\oplus"] \arrow[r, bend right=50, swap, <->]&A_2 \arrow[r,phantom, "\oplus"] \arrow[r, bend left=50, above, "\tau^*", <->] &A_2 \arrow[r,phantom, "\oplus"]  
&A_2 \arrow[r,phantom, "\oplus"] \arrow[r, bend right=50, swap, "\varphi^*", <->] 
&A_2 \arrow[r,phantom, "\oplus"]  
\arrow[l, start anchor=east, end anchor=west, above, no head, decorate, decoration={brace,mirror,amplitude=5pt}, "\tau^*=id" inner sep=5pt, yshift=0.8em, swap]
&A_2 \arrow[r,phantom, "\oplus"] \arrow[r, bend left=50, above, "\tau^*", <->] 
\arrow[r, bend right=50, swap, "\varphi^*", <->] 
&A_2 \arrow[r,phantom, "\oplus"]  
&A_2(2)\arrow[start anchor=east, end anchor=west, no head, decorate, decoration={brace,mirror,amplitude=5pt}, "\tau^*=-id" inner sep=5pt, yshift=0.8em, swap] \arrow[start anchor=west, end anchor=east, no head, decorate, decoration={brace,mirror,amplitude=5pt}, "\varphi^*=id" inner sep=5pt, yshift=-1em, swap]
\arrow[r,phantom, "\oplus"] &U(3) \arrow[r,phantom, "\oplus"] \arrow[r, start anchor=west, end anchor=east, no head, decorate, decoration={brace,mirror,amplitude=5pt}, "\tau^*=\varphi^*=id" inner sep=5pt, yshift=-1.5em, swap]&\begin{bmatrix}4 &2\\2 &4\end{bmatrix}
\end{tikzcd}.\]
\item For $i=1,2$, denote ${a_i,\ b_i,\ c_i,\ d_i,\ e_i,\ f_i,\ g_i,\ h_i}$ the generators of the eight copies of $A_2$, such that
\begin{align*}
\tau^*&:a_i \leftrightarrow b_i,\quad c_i \leftrightarrow d_i,\quad g_i \leftrightarrow h_i;\\
\varphi^*&:a_i \leftrightarrow d_i,\quad b_i \leftrightarrow c_i,\quad e_i \leftrightarrow f_i; \quad g_i \leftrightarrow h_i;
\end{align*}
denote $w,z$ the generators of $A_2(2)$ (on which $\tau^*$ acts as $-id$ and $\varphi^*$ as $id$), $x,y$ the generators of $U(3)$, and $v_1,v_2$ the generators of $\begin{bmatrix}4 &2\\2 &4\end{bmatrix}$. Then the lattice $\Lambda_{\mathrm{K3}}$ is isomorphic to the overlattice $H^2(X,\mathbb Z)$ of $W$ obtained by adding the following elements to the list of generators (and the action of $\tau^*$ and $\varphi^*$ on them is deduced by $\mathbb Q$-linear extension):
\begin{align}\label{generatori}
\alpha&=(-a_1+a_2+d_1-d_2-e_1+e_2+f_1-f_2-g_1+g_2+h_1-h_2)/3, 
\nonumber\\
\beta&=(-b_1+b_2+c_1-c_2-e_1+e_2+f_1-f_2+g_1-g_2-h_1+h_2)/3, 
\nonumber\\
\gamma&=(x-y-e_1+e_2-f_1+f_2)/3, 
\nonumber\\
\delta&=(x-c_1+c_2-d_1+d_2-e_1+e_2)/3, 
\\
\varepsilon&=(x-z+w+c_1-c_2-e_1+e_2-g_1+g_2+h_1-h_2)/3, 
\nonumber\\
\zeta&=(x+z+c_1+c_2+e_1+e_2+g_1+g_2+h_1+h_2+\varepsilon)/2+v_2/2, 
\nonumber\\
\eta&=(x+c_1+c_2+e_1+e_2+\varepsilon)/2+(g_1-g_2+h_1-h_2)/6+v_1/6-v_2/3.
\nonumber
\end{align}
\end{enumerate}
\end{proposition}

\begin{proof}
The first five elements in \eqref{generatori} are the same as in \eqref{generatoriNS}, while $\zeta,\eta$ are obtained by the (unique) description of $H^2(X_\omega,\mathbb Z)$ as overlattice of $NS(X_\omega)\oplus T_{X_\omega}$.
\end{proof}

\subsection{Invariant and co-invariant lattices}\label{sec:azioni}

The group $(\mathbb Z/2\mathbb Z)^2$ acts symplectically in a unique way on the second integral cohomology lattice of a K3 surface \cite[Thm. 4.7]{Nikulin2}: the invariant and  co-invariant lattices for this action can be found in \cite[Prop. 4.3]{GS1}. Here we provide an explicit embedding of them in $H^2(X,\mathbb Z)$.

The invariant lattice $\Lambda_{\mathrm{K3}}^{\langle\tau,\varphi\rangle}$ is an overlattice of the lattice $I=A_2(4)\oplus A_2(2)^{\oplus 2}\oplus U(3)\oplus\begin{bmatrix} 4 & 2 \\ 2 & 4\end{bmatrix}=\langle a_1+b_1+c_1+d_1,\ a_2+b_2+c_2+d_2,\ e_1+f_1,\ e_2+f_2,\ g_1+h_1,\ g_2+h_2,\ x,\ y,\ v_1,\ v_2\rangle$ obtained by adding as generators the elements
$(v_1+v_2+g_1+h_1-g_2-h_2)/3,\ (a_1+b_1+c_1+d_1-(a_2+b_2+c_2+d_2)+e_1+f_1-e_2-f_2+x)/3, \gamma$; the co-invariant lattice $\Omega_{2,2}$ is an overlattice of the lattice
 \[ \Delta=\left[ \begin{tabular}{ c | c | c  }
  $A_2(2)^{\oplus 3}$ &0 & 0 \TBstrut\\
  \hline
  0 &$A_3(2)$ & $A_3$ \TBstrut\\
  \hline
  0 &$A_3$ & $A_3(2)$\TBstrut\\
\end{tabular} \right]\]
spanned over $\mathbb Z$ by $\{z,\ w,\ f_1-e_1,\ f_2-e_2,\ h_1-g_1,\ h_2-g_2,\ b_1-a_1,\ a_1-c_1,\ c_1-d_1,\ a_2-b_2,\ c_2-a_2,\ d_2-c_2\}$,
obtained by adding as generators the elements
$(a_1-b_1-c_1+d_1-a_2+b_2+c_2-d_2+z-w)/3,\ (-a_1+d_1+a_2-d_2-e_1+f_1+e_2-f_2-g_1+h_1+g_2-h_2)/3,\ (a_1+b_1-c_1-d_1-a_2-b_2+c_2+d_2-e_1+f_1+e_2-f_2)/3$. The discriminant group of $\Omega_{2,2}$ is $(\mathbb Z/2\mathbb Z)^6\times(\mathbb Z/4\mathbb Z)^2$. 
\begin{remark}
The lattice $\Omega_{2,2}$ contains three different copies of the lattice $\Omega_2\simeq E_8(2)$, co-invarant for $\tau,\ \varphi$ and $\rho$. 
\end{remark}

\section{Quotients}\label{sec:quotients_Klein}

For each of the abelian groups $G$ that act symplectically on a K3 surface $X$, Nikulin provides in \cite[\S 5-7]{Nikulin2} a description of the singular locus of the quotient surface $Y=X/G$, and of the \emph{exceptional lattice} $M_G$: this is the minimal primitive sublattice of $\Lambda_{\mathrm{K3}}$ containing all the exceptional curves of the minimal resolution $\tilde Y$ of $Y$. Denoting $q: X\rightarrow Y$ the quotient map, $H^2(\tilde Y,\mathbb Z)$ is an overlattice of finite index of $q_*H^2(X,\mathbb Z)\oplus M_G$.

\begin{remark}\label{unimodularity}
By \cite[Prop. 1.4.1.a]{Nikulin1}, since $H^2(\tilde Y,\mathbb Z)$ is unimodular and unique in its genus, there is essentially only one way to obtain it as overlattice of finite index of $q_*H^2(X,\mathbb Z)\oplus M_G$. This holds for $\Lambda_{\mathrm{K3}}$ more in general, given lattices $A,B$ in direct sum such that $sign(A)+sign(B)=(3,19)$ and their discriminant forms satisfy $q_A=-q_B$.
\end{remark}

\begin{definition}[\protect{\cite[Def. 6.2, case 1a]{Nikulin2}}] \label{NikulinLattice}
Denote \emph{Nikulin lattice} the lattice $N\coloneqq M_{\mathbb Z/2\mathbb Z}$: given $\{n_1,\dots , n_8\}$ orthogonal $(-2)$-classes, then a set of generators over $\mathbb Z$ for $N$ is $\{n_1,\dots , n_8, \nu\}$, where ${\nu=(n_1+\dots +n_8)/2}$.
\end{definition}

The symplectic action of $(\mathbb Z/2\mathbb Z)^2=\{1,\tau,\varphi,\rho\}$ on a K3 surface $X$ gives 24 isolated points with nontrivial stabilizer. Call $Fix_\tau=\{t_1,\dots,t_8\},\ Fix_\varphi=\{q_1,\dots,q_8\}, Fix_\rho=\{r_1,\dots,r_8\}$: then $\tau$ and $\varphi$ act on $Fix_\rho$ as the same permutation $(r_1,r_2)(r_3,r_4)(r_5,r_6)$ $(r_7,r_8)$; $\tau$ and $\rho$ act on $Fix_\varphi$ as $(q_1,q_2)(q_3,q_4)(q_5,q_6)(q_7,q_8)$, $\rho$ and $\varphi$ act on $Fix_\tau$ as $(t_1,t_2)(t_3,t_4)(t_5,t_6)(t_7,t_8)$.

Consider the quotient surfaces $Y=X/(\mathbb Z/2\mathbb Z)^2, \ Z_\tau=X/\tau$; 

resolve the singularities to obtain the K3 surfaces $\tilde Y,\ \tilde Z_\tau$: then $\varphi$ induces an involution $\hat\varphi$ on $Z_\tau$, and this involution can be extended to $\tilde Z_\tau$. The surfaces $\tilde Y$ and $\widetilde{\tilde Z_\tau/\hat\varphi}$ are isomorphic, as they are birationally equivalent K3 surfaces (and the same holds exchanging the roles of $\tau$ and $\varphi$). In the following sections, we are going to describe the maps induced in cohomology by those in the following diagram:
\begin{equation}\label{diagrammaquozienti}
\xymatrix{
\ &X \ar@{-->}[dr]^{\pi_{\varphi}}  \ar@{-->}[dl]_{\pi_{\tau}} &\ \\
\tilde Z_\tau \ar@{-->}[dr]_{\widehat{\pi_\varphi}}&\  &\tilde Z_\varphi \ar@{-->}[dl]^{\widehat{\pi_\tau}}\\
\ &\tilde Y &\
}\end{equation}

\subsection{The map $\pi_{\tau*}$ and the surface $\tilde Z_\tau$}\label{sec:pitau}

\begin{proposition}\label{pitau}
The map $\pi_{\tau*}$ acts in the following way on the sublattice $W$ of $H^2(X,\mathbb Z)$:\\
\[\adjustbox{scale=0.7,center}{\begin{tikzcd}[column sep=0.001pt] 
\begin{matrix}A_2\\ a_1,a_2\end{matrix} \arrow[dr]  &\begin{matrix}\oplus \\ \ \end{matrix}
&\begin{matrix}A_2\\ b_1,b_2\end{matrix}\arrow[dl]  &\begin{matrix}\oplus \\ \ \end{matrix} 
&\begin{matrix}A_2\\ c_1,c_2\end{matrix}\arrow[dr]  &\begin{matrix}\oplus \\ \ \end{matrix}
&\begin{matrix}A_2\\ d_1,d_2\end{matrix}\arrow[dl]  &\begin{matrix}\oplus \\ \ \end{matrix}
&\begin{matrix}A_2\\ e_1,e_2\end{matrix}\arrow[d] &\begin{matrix}\oplus \\ \ \end{matrix} 
&\begin{matrix}A_2\\ f_1,f_2\end{matrix} \arrow[d] &\begin{matrix}\oplus \\ \ \end{matrix} 
&\begin{matrix}A_2\\ g_1,g_2\end{matrix}\arrow[dr] &\begin{matrix}\oplus \\ \ \end{matrix}
&\begin{matrix}A_2\\ h_1,h_2\end{matrix} \arrow[dl] &\begin{matrix}\oplus \\ \ \end{matrix}
&\begin{matrix}A_2(2)\\ z,w\end{matrix}\arrow[d]  &\begin{matrix}\oplus \\ \ \end{matrix} 
&\begin{matrix}U(3)\\ x,y\end{matrix} \arrow[d]
&\begin{matrix}\oplus \\ \ \end{matrix} 
&\begin{matrix}\begin{bmatrix}4 &2\\2 &4\end{bmatrix}\\ v_1,v_2\end{matrix} \arrow[d] \\
\ &\begin{matrix}A_2\\ \hat a_1,\hat a_2\end{matrix}\arrow[rrrr,phantom, "\begin{matrix}\oplus \\ \ \end{matrix}"]  &\ &\ &\ 
&\begin{matrix}A_2\\ \hat c_1,\hat c_2\end{matrix}\arrow[rrr,phantom, "\begin{matrix}\oplus \\ \ \end{matrix}"]  &\ &\ 
&\begin{matrix}A_2(2)\\ \hat e_1,\hat e_2\end{matrix}\arrow[rr,phantom, "\begin{matrix}\oplus \\ \ \end{matrix}"]  &\  
&\begin{matrix}A_2(2)\\ \hat f_1,\hat f_2\end{matrix}\arrow[rrr,phantom, "\begin{matrix}\oplus \\ \ \end{matrix}"]  &\ &\ 
&\begin{matrix}A_2\\ \hat g_1,\hat g_2\end{matrix}\arrow[rrr,phantom, "\begin{matrix}\oplus \\ \ \end{matrix}"]  &\ &\ 
&\begin{matrix} 0\\ \ \end{matrix}\arrow[rr,phantom, "\begin{matrix}\oplus \\ \ \end{matrix}"]  &\  
&\begin{matrix}U(6)\\ \hat x,\hat y\end{matrix}\arrow[rr,phantom, "\begin{matrix}\oplus \\ \ \end{matrix}"]  &\ 
&\begin{matrix}\begin{bmatrix}8 &4\\4 &8\end{bmatrix}\\ \hat v_1,\hat v_2\end{matrix}
\end{tikzcd}}\]
\end{proposition}
\begin{proof}
The action of $\tau^*$ on $W$ is described in Proposition \ref{sublattice}: 
we can use it to compute the intersection form of $\pi_{\tau*}W$ via the push-pull formula.
Since $\pi_{\tau}$ is a finite morphism of degree 2, for any $x_1,x_2\in W$ we get
\[\pi_{\tau*}x_1\cdot\pi_{\tau*}x_2=\frac{1}{2}(\pi^{*}_\tau\pi_{\tau*}x_1\cdot\pi^{*}_\tau\pi_{\tau*}x_2)\]
where $\pi^{*}_\tau\pi_{\tau*}x_1=x_1+\tau^*x_1$. Therefore, if $\tau^*$ exchanges two copies of $A_2$, $\pi_{\tau*}(A_2\oplus A_2)=A_2$; if $\tau^*$ acts as the identity on a lattice $L$, then $\pi_{\tau*}L=L(2)$; if $\tau^*$ acts as $-id$ on a lattice $L$, then $\pi_{\tau*}L=0$.
\end{proof}

We construct $H^2(\tilde Z_\tau, \mathbb Z)$ as overlattice of $N\oplus\pi_{\tau*}H^2(X,\mathbb Z)$: by Remark \ref{unimodularity} any way is equivalent up to isometries of $\Lambda_{\mathrm{K3}}$. For convenience, we give here the list of generators we're going to use in the following computations:
\begin{align}\label{generatoriZtau}
s_1&=(\hat c_1-\hat c_2+\hat e_2+\hat f_2+\hat \gamma-\hat \varepsilon+n_5-n_8+n_3+n_2)/2-n_8,\nonumber\\
s_2&=(\hat a_1-\hat a_2-\hat \alpha+\hat f_1+\hat f_2-\hat \varepsilon+n_4-n_8+n_3+n_2)/2-n_8 ,\nonumber\\
s_3&=(\hat e_2+\hat f_1+n_7+n_5+n_4+n_3)/2-2n_8 ,\\
s_4&=(\hat c_1-\hat c_2+\hat e_2+\hat f_1-\hat \varepsilon+n_7-n_8+n_5+n_4)/2-n_8,\nonumber\\
s_5&=(\hat a_1-\hat a_2+\hat c_1-\hat c_2-\hat \alpha+n_6+n_5+n_4+n_2)/2-2n_8,\nonumber\\    
s_6&=(\hat a_1-\hat a_2+\hat c_1-\hat c_2-\hat \alpha+\hat f_1+n_7-n_8+n_6+n_3)/2-n_8.\nonumber
\end{align}

\begin{remark}\label{R22}
The lattice $\pi_{\tau*}\Omega_{2,2}$ is isomorphic to $D_4(2)$ with the following generators: $d_1=(\hat e_2-\hat f_2+\hat f_1-\hat e_1+\hat c_1-\hat a_1-\hat c_2+\hat a_2)/3-\hat f_1+\hat e_1,\ d_2=(\hat e_2-\hat f_2+\hat f_1-\hat e_1+\hat c_1-\hat a_1-\hat c_2+\hat a_2)/3,\ d_3=\hat a_1-\hat c_1,\ d_4=\hat c_1-\hat a_1+\hat c_2-\hat a_2$.
\end{remark}

\begin{definition}\label{def:Gammatau}
Define the lattice $\Gamma_\tau$ as the lattice of rank 12 obtained as primitive completion of $\pi_{\tau*}\Omega_{2,2}\oplus N$ in $H^2(\tilde Z_\tau,\mathbb Z)$.
With the latter constructed as in \eqref{generatoriZtau}, $\Gamma_\tau$ is obtained as overlattice of $\pi_{\tau*}\Omega_{2,2}\oplus N$ by adding as generators the elements
\[
x_1=(d_4-d_2+n_2+n_4+n_5+n_6)/2,\quad
x_2=(d_1-d_2+n_3+n_7+n_2+n_5)/2.\]
\end{definition}

\subsection{The map $\pi_{\varphi*}$ and the surface $\tilde Z_\varphi$}\label{sec:piphi}

The action of $\varphi^*$ on the sublattice $W$ of $H^2(X,\mathbb Z)$ is different that that of $\tau^*$: we provide an analogue to Proposition \ref{pitau} for the quotient map $\pi_{\varphi*}$.

\begin{proposition}\label{piphi}
The map $\pi_{\varphi*}$ acts in the following way on the sublattice $W$ of $H^2(X,\mathbb Z)$:\\
\[\adjustbox{scale=0.7,center}{\begin{tikzcd}[column sep=0.001pt] 
\begin{matrix}A_2\\ a_1,a_2\end{matrix} \arrow[dr]  &\begin{matrix}\oplus \\ \ \end{matrix}
&\begin{matrix}A_2\\ d_1,d_2\end{matrix}\arrow[dl]  &\begin{matrix}\oplus \\ \ \end{matrix} 
&\begin{matrix}A_2\\ b_1,b_2\end{matrix}\arrow[dr]  &\begin{matrix}\oplus \\ \ \end{matrix}
&\begin{matrix}A_2\\ c_1,c_2\end{matrix}\arrow[dl]  &\begin{matrix}\oplus \\ \ \end{matrix}
&\begin{matrix}A_2\\ e_1,e_2\end{matrix}\arrow[dr]  &\begin{matrix}\oplus \\ \ \end{matrix}
&\begin{matrix}A_2\\ f_1,f_2\end{matrix}\arrow[dl]  &\begin{matrix}\oplus \\ \ \end{matrix}
&\begin{matrix}A_2\\ g_1,g_2\end{matrix}\arrow[dr] &\begin{matrix}\oplus \\ \ \end{matrix}
&\begin{matrix}A_2\\ h_1,h_2\end{matrix} \arrow[dl] &\begin{matrix}\oplus \\ \ \end{matrix}
&\begin{matrix}A_2(2)\\ z,w\end{matrix}\arrow[d]  &\begin{matrix}\oplus \\ \ \end{matrix} 
&\begin{matrix}U(3)\\ x,y\end{matrix} \arrow[d]
&\begin{matrix}\oplus \\ \ \end{matrix} 
&\begin{matrix}\begin{bmatrix}4 &2\\2 &4\end{bmatrix}\\ v_1,v_2\end{matrix} \arrow[d] \\
\ &\begin{matrix}A_2\\ \tilde a_1,\tilde a_2\end{matrix}\arrow[rrrr,phantom, "\begin{matrix}\oplus \\ \ \end{matrix}"]  &\ &\ &\ 
&\begin{matrix}A_2\\ \tilde b_1,\tilde b_2\end{matrix}\arrow[rrr,phantom, "\begin{matrix}\oplus \\ \ \end{matrix}"]  &\ &\  &\
&\begin{matrix}A_2\\ \tilde e_1,\tilde e_2\end{matrix}\arrow[rrr,phantom, "\begin{matrix}\oplus \\ \ \end{matrix}"]  &\ &\ &\
&\begin{matrix}A_2\\ \tilde g_1,\tilde  g_2\end{matrix}\arrow[rrr,phantom, "\begin{matrix}\oplus \\ \ \end{matrix}"]  &\ &\ 
&\begin{matrix}A_2(4)\\ \tilde z,\tilde w\end{matrix}\arrow[rr,phantom, "\begin{matrix}\oplus \\ \ \end{matrix}"]  &\ 
&\begin{matrix}U(6)\\ \tilde x,\tilde y\end{matrix}\arrow[rr,phantom, "\begin{matrix}\oplus \\ \ \end{matrix}"]  &\ 
&\begin{matrix}\begin{bmatrix}8 &4\\4 &8\end{bmatrix}\\ \tilde v_1,\tilde v_2\end{matrix}
\end{tikzcd}}\]
\end{proposition}

As above, we give the list of generators we used to construct $H^2(\tilde Z_\varphi, \mathbb Z)$ as overlattice of $N\oplus\pi_{\varphi*}H^2(X,\mathbb Z)$:
\begin{align}\label{generatoriZphi}
t_1&=(\tilde b_1+\tilde e_2+\tilde g_1+\tilde g_2-\tilde\varepsilon+\tilde \gamma+\tilde y+\tilde\eta)/2+(n_2+n_3+n_5+n_8)/2;\nonumber\\
t_2&=(\tilde a_1+\tilde a_2+\tilde\delta+\tilde z-\tilde\varepsilon)/2+(n_2+n_3+n_4+n_8)/2;\nonumber\\
t_3&=(\tilde\gamma+\tilde y)/2+(n_3+n_4+n_5+n_7)/2;\\          
t_4&=(\tilde a_1+\tilde a_2+\tilde\delta+\tilde\varepsilon+\tilde y)/2+(n_4+n_5+n_7+n_8)/2;\nonumber\\
t_5&=(\tilde b_1+\tilde e_2+\tilde g_1+\tilde g_2+\tilde\varepsilon+\tilde\zeta)/2+(n_2+n_4+n_5+n_6)/2;\nonumber\\ 
t_6&=(\tilde a_1+\tilde a_2+\tilde\delta+\tilde\varepsilon)/2+(n_3+n_6+n_7+n_8)/2.\nonumber
\end{align}

\begin{remark}
The lattice $\pi_{\varphi*}\Omega_{2,2}$ is isomorphic to $D_4(2)$ with the following generators: $d'_1=(2\tilde b_2-2\tilde a_2+\tilde z-\tilde w+\tilde b_1-\tilde a_1)/3,\ d'_2=\tilde w+(2\tilde b_2-2\tilde a_2+\tilde z-\tilde w+\tilde b_1-\tilde a_1)/3,\ d'_3=\tilde a_2-\tilde b_2,\ d'_4=\tilde a_1-\tilde b_1$.
\end{remark}
\begin{definition}
Define the lattice $\Gamma_\varphi$ as the overlattice of $N\oplus\pi_{\varphi*}\Omega_{2,2}$ obtained by adding to the set of generators the elements
\[x'_1=(d'_2-d'_1+n_2+n_3+n_4+n_8)/2,\quad x'_2=(d'_2+d'_4+n_3+n_6+n_7+n_8)/2;\]
\end{definition}
\begin{lemma}\label{prop:Gamma_22phi}
The lattices $\Gamma_{\tau}$ and $\Gamma_\varphi$ are isomorphic.
\end{lemma}
\begin{proof}
The lattices $\pi_{\tau*}\Omega_{2,2}$ and $\pi_{\varphi*}\Omega_{2,2}$ are both isomorphic to $D_4(2)$. Moreover, the gluings that realizes $\Gamma_{\tau}$ as an overlattice of $\pi_{\tau*}\Omega_{2,2}\oplus N$, and $\Gamma_\varphi$ as an overlattice of $\pi_{\varphi*}\Omega_{2,2}\oplus N$, are isomorphic: indeed, one can easily check that the orbits for the action of $O(D_4(2))$ on $A_{D_4(2)}$, and of $O(N)$ on $A_N$, are determined by the order and square of their elements. 
\end{proof}

\begin{definition}\label{def:Gamma22}
Consider the symplectic action of $G=(\mathbb Z/2\mathbb Z)^2$ on a K3 surface $X$, let $\iota\in G$ be any involution and consider $\tilde Z$ the resolution of singularities of $X/\iota$. Define the lattice $\Gamma_{2,2}$ as the primitive completion of $N\oplus\pi_{\iota*}\Omega_{2,2}$ in $H^2(\tilde Z,\mathbb Z)$.
\end{definition}

\subsection{The surface $\tilde Y$ as quotient of $\tilde Z_\tau$}\label{sec:hatrho}

We conclude with the description of the K3 surface $\tilde Y$, which is the minimal resolution of the quotient $X/(\mathbb Z/2\mathbb Z)^2$. We obtain $\tilde Y$ as minimal resolution of the quotient of either $\tilde Z_\tau, \tilde Z_\varphi$ by the residual symplectic involution $\hat\varphi,\hat\tau$ respectively.

The residual involution $\hat\varphi$ fixes eight isolated points on $\tilde Z_\tau$, and it acts on the exceptional curves introduced by the resolution $\tilde Z_\tau\rightarrow Z_\tau$ (which are represented by the classes $n_1,\dots, n_8$ in $NS(\tilde Z_\tau)$) by exchanging them pairwise.
\begin{proposition}\label{pirho}
Consider the sublattice $\pi_{\tau*}W\oplus A_1^{\oplus 8}$ of finite index of $H^2(\tilde Z_\tau, \mathbb Z)$: the map $ \widehat{\pi_{\varphi}}_*$ acts in the following way on it:\\
\[\adjustbox{scale=0.7,center}{\begin{tikzcd}[column sep=0.001pt] 
\begin{matrix}A_2\\ \hat a_1,\hat a_2\end{matrix} \arrow[dr]  &\begin{matrix}\oplus \\ \ \end{matrix}
&\begin{matrix}A_2\\ \hat c_1,\hat c_2\end{matrix}\arrow[dl]  &\begin{matrix}\oplus \\ \ \end{matrix} 
&\begin{matrix}A_2(2)\\ \hat e_1,\hat e_2\end{matrix}\arrow[dr]  &\begin{matrix}\oplus \\ \ \end{matrix}
&\begin{matrix}A_2(2)\\ \hat f_1,\hat f_2\end{matrix}\arrow[dl]  &\begin{matrix}\oplus \\ \ \end{matrix}
&\begin{matrix}A_2\\ \hat g_1,\hat g_2\end{matrix}\arrow[d] &\begin{matrix}\oplus \\ \ \end{matrix} 
&\begin{matrix}U(6)\\ \hat x,\hat y\end{matrix} \arrow[d]
&\begin{matrix}\oplus \\ \ \end{matrix} 
&\begin{matrix}\begin{bmatrix}8 &4\\4 &8\end{bmatrix}\\ \hat v_1,\hat v_2\end{matrix} \arrow[d] &\begin{matrix}\oplus \\ \ \end{matrix}
&\begin{matrix}A_1^{\oplus 2} \\ n_1,n_8 \end{matrix} \arrow[d] &\begin{matrix}\oplus \\ \ \end{matrix}
&\begin{matrix}A_1^{\oplus 2} \\ n_2,n_5 \end{matrix} \arrow[d]  &\begin{matrix}\oplus \\ \ \end{matrix}
&\begin{matrix}A_1^{\oplus 2} \\ n_3,n_7 \end{matrix} \arrow[d] &\begin{matrix}\oplus \\ \ \end{matrix}
&\begin{matrix}A_1^{\oplus 2} \\ n_4,n_6 \end{matrix} \arrow[d]\\
\ &\begin{matrix}A_2\\ \overline a_1,\overline a_2\end{matrix}\arrow[rrrr,phantom, "\begin{matrix}\oplus \\ \ \end{matrix}"]  &\ &\ &\ 
&\begin{matrix}A_2(2)\\ \overline e_1,\overline e_2\end{matrix}\arrow[rrr,phantom, "\begin{matrix}\oplus \\ \ \end{matrix}"]  &\ &\ 
&\begin{matrix}A_2(2)\\ \overline g_1,\overline g_2\end{matrix}\arrow[rr,phantom, "\begin{matrix}\oplus \\ \ \end{matrix}"]  &\  
&\begin{matrix}U(12)\\ \overline x,\overline y\end{matrix}\arrow[rr,phantom, "\begin{matrix}\oplus \\ \ \end{matrix}"]  &\ 
&\begin{matrix}\begin{bmatrix}16 &8\\8 &16\end{bmatrix}\\ \overline v_1,\overline v_2\end{matrix} &\begin{matrix}\oplus \\ \ \end{matrix}
&\begin{matrix}A_1\\ \overline n_1\end{matrix}&\begin{matrix}\oplus \\ \ \end{matrix}
&\begin{matrix}A_1\\ \overline n_2\end{matrix}&\begin{matrix}\oplus \\ \ \end{matrix}
&\begin{matrix}A_1\\ \overline n_3\end{matrix}&\begin{matrix}\oplus \\ \ \end{matrix}
&\begin{matrix}A_1\\ \overline n_4\end{matrix}
\end{tikzcd}}\]
The lattice $\widehat{\pi_{\varphi}}_*H^2(\tilde Z_\tau, \mathbb Z)$ can be obtained by $\mathbb Q$-linear extension to the elements $\hat\alpha,\hat\gamma,\hat\varepsilon,\hat\zeta,\hat\eta$ which are the image via $\pi_{\tau*}$ of the elements in \eqref{generatori}, and $\nu, s_1,\dots, s_6$ defined in \eqref{generatoriZtau}. The symbol $\overline{\star}$ denotes the image of $\star$ in $\widehat{\pi_{\varphi}}_*H^2(\tilde Z_\tau, \mathbb Z)$. 
\end{proposition}
\begin{proof}
The only difficult thing is to determine which are the pairs of classes exchanged by $\widehat{\pi_{\varphi}}_*|_{A_1^{\oplus 8}}$. To do this, we need to ensure that the the intersection form of the images of $s_1,\dots, s_6$ via $\widehat{\pi_{\varphi}}_*$  computed with the  push-pull formula, is that of an integral even lattice: the only valid choice is the one in the statement. 
\end{proof}

\begin{remark}\label{Omega2inGamma22}
Since $\hat\varphi$ is a symplectic involution on $\tilde Z_\tau$, its co-invariant lattice is a copy of $\Omega_2=E_8(2)$: this is entirely contained in $\Gamma_{2,2}$, as the orthogonal complement of $\langle n_1+n_8, n_2+n_5, n_3+n_7, n_4+n_6, \hat e_1+\hat f_1,\hat e_2+\hat f_2, \hat a_1+\hat c_1, \hat a_2+\hat c_2 \rangle$.
\end{remark}

\begin{remark}\label{gammamezzi}
Consider the map $(\pi_{2,2})_*$, defined as the composition $\widehat{\pi_{\varphi}}_*\circ\pi_{\tau*}$: then $(\pi_{2,2})_*H^2(X,\mathbb Z)$ is a sublattice of index 2 of $\widehat{\pi_{\varphi}}_*H^2(\tilde Z_\tau, \mathbb Z)$. Indeed, it does not contain the element $\overline\gamma/2$, that is integral in $\widehat{\pi_{\varphi}}_*H^2(\tilde Z_\tau, \mathbb Z)$: in fact, it holds
\[\widehat{\pi_{\varphi}}_*(s_1+s_3+s_4)=\overline\gamma/2+(\overline a_1-\overline a_2+\overline e_1+2\overline e_2-\overline\varepsilon+2\overline n_2+\overline n_4+2\overline n_3-5\overline n_1).\]
\end{remark}

The resolution of the singularities $\tilde Y\rightarrow \tilde Z_\tau/\hat\varphi$ introduces in cohomology another copy of the lattice $N$. Calling $m_1,\dots m_8$ the $(-2)$-classes that generate $N$ over $\mathbb Q$, we construct $H^2(\tilde Y,\mathbb Z)$ as overlattice of $\widehat{\pi_{\varphi}}_*H^2(\tilde Z_\tau, \mathbb Z)\oplus N$ using the following elements:
\begin{align}\label{generatoriY}
k_1&=(\overline a_2+\overline e_1+\overline g_2+\overline \eta)/2+(m_2+m_3+m_5+m_8)/2,\nonumber \\
k_2&=(\overline g_1+\overline\eta+\overline\zeta)/2+(m_2+m_3+m_4+m_8)/2,\nonumber \\
k_3&=(\overline a_1+\overline a_2+\overline g_1+\overline s_1+\overline \varepsilon+\overline s_3+\overline s_4+\overline n_8)/2+(m_3+m_4+m_5+m_7)/2, \\
k_4&=(\overline a_2+\overline e_1+\overline \zeta)/2+(m_4+m_5+m_7+m_8)/2,\nonumber \\
k_5&=(\overline a_2+\overline e_2+\overline s_1+\overline\zeta+\overline s_3+\overline s_4+\overline n_8)/2+(m_2+m_4+m_5+m_6)/2,\nonumber \\
k_6&=(\overline a_1+\overline e_2+\overline \varepsilon+\overline \zeta)/2+(m_3+m_6+m_7+m_8)/2.\nonumber
\end{align}

\subsection{The exceptional lattice $M_{2,2}$ and the map $\pi_{2,2*}$}\label{sec:M22}

The lattice $M_{2,2}$, as described in \cite[\S 6, case 2a]{Nikulin2}, is an overlattice of $A_1^{\oplus 12}=\langle v_1,\dots, v_{12}\rangle$ obtained by adding as generator the elements $(v_1+\dots+v_8)/2$ and $(v_5+\dots v_{12})/2$. An explicit embedding of $M_{2,2}$ in $H^2(\tilde Y,\mathbb Z)$ constructed as in \eqref{generatoriY}, as overlattice of $N\oplus\widehat{\pi_{\rho}}_*\Gamma_{2,2}=N\oplus\langle\overline n_1,\overline n_2, \overline n_3, \overline n_4\rangle$, can be obtained as follows.

\begin{proposition}\label{prop:M22}
The lattice $M_{2,2}$ is generated over $\mathbb Q$ by the elements $\overline n_1,\dots,\overline n_4,$ $m_1,\dots, m_8$. To get a set of $\mathbb Z$-generators, add the elements
\[\mu_1=\frac{m_1+\dots+ m_8}{2},\quad\mu_2=\frac{\overline n_1+\overline n_2+\overline n_3+\overline n_4+m_1+m_2+m_7+m_8}{2}.\]
\end{proposition}
\begin{proof}
A $\mathbb Q$-basis of $M_{2,2}$ is $\{\overline n_1,\dots,\overline n_4,m_1,\dots, m_8\}$, as these are the classes that come from resolution of the singularities in our construction. Notice moreover that it holds $\mu_1\in N$, while $\mu_2$ is the only other linear combination of the form $(\overline n_1+\overline n_2+\overline n_3+\overline n_4+m_i+m_j+m_h+m_k)/2$ which is integral in $H^2(\tilde Y,\mathbb Z)$ and independent from $\mu_1$.
\end{proof}

\subsection{The surface $\tilde Y$ as quotient of $\tilde Z_\varphi$}\label{sec:hatpsi}

\begin{proposition}\label{pihattau}
Consider the sublattice $\pi_{\varphi*}W\oplus A_1^{\oplus 8}$ of finite index of $H^2(\tilde Z_\varphi, \mathbb Z)$: the map $ \widehat{\pi_{\tau}}_*$ acts in the following way on it:\\
\[\adjustbox{scale=0.7,center}{\begin{tikzcd}[column sep=0.001pt] 
\begin{matrix}A_2\\ \tilde a_1,\tilde a_2\end{matrix} \arrow[dr]  &\begin{matrix}\oplus \\ \ \end{matrix}
&\begin{matrix}A_2\\ \tilde b_1,\tilde b_2\end{matrix}\arrow[dl]  &\begin{matrix}\oplus \\ \ \end{matrix} 
&\begin{matrix}A_2\\ \tilde e_1,\tilde e_2\end{matrix}\arrow[d]  &\begin{matrix}\oplus \\ \ \end{matrix}
&\begin{matrix}A_2\\ \tilde g_1,\tilde g_2\end{matrix}\arrow[d]  &\begin{matrix}\oplus \\ \ \end{matrix}
&\begin{matrix}A_2(2)\\ \tilde w,\tilde z\end{matrix}\arrow[d] &\begin{matrix}\oplus \\ \ \end{matrix} 
&\begin{matrix}U(6)\\ \hat x,\hat y\end{matrix} \arrow[d]
&\begin{matrix}\oplus \\ \ \end{matrix} 
&\begin{matrix}\begin{bmatrix}8 &4\\4 &8\end{bmatrix}\\ \hat v_1,\hat v_2\end{matrix} \arrow[d] &\begin{matrix}\oplus \\ \ \end{matrix}
&\begin{matrix}A_1^{\oplus 2} \\ n_1,n_5 \end{matrix} \arrow[d] &\begin{matrix}\oplus \\ \ \end{matrix}
&\begin{matrix}A_1^{\oplus 2} \\ n_2,n_4 \end{matrix} \arrow[d]  &\begin{matrix}\oplus \\ \ \end{matrix}
&\begin{matrix}A_1^{\oplus 2} \\ n_3,n_8\end{matrix} \arrow[d] &\begin{matrix}\oplus \\ \ \end{matrix}
&\begin{matrix}A_1^{\oplus 2} \\ n_6,n_7 \end{matrix} \arrow[d]\\
\ &\begin{matrix}A_2\\ \overline a_1,\overline a_2\end{matrix}\arrow[rrr,phantom, "\begin{matrix}\oplus \\ \ \end{matrix}"]  &\ &\ 
&\begin{matrix}A_2(2)\\ \overline e_1,\overline e_2\end{matrix}\arrow[rr,phantom, "\begin{matrix}\oplus \\ \ \end{matrix}"]  &\ 
&\begin{matrix}A_2(2)\\ \overline g_1,\overline g_2\end{matrix}\arrow[rr,phantom, "\begin{matrix}\oplus \\ \ \end{matrix}"]  &\ &\begin{matrix} 0\\ \ \end{matrix} \arrow[rr,phantom, "\begin{matrix}\oplus \\ \ \end{matrix}"] &\ 
&\begin{matrix}U(12)\\ \overline x,\overline y\end{matrix}\arrow[rr,phantom, "\begin{matrix}\oplus \\ \ \end{matrix}"]  &\
&\begin{matrix}\begin{bmatrix}16 &8\\8 &16\end{bmatrix}\\ \overline v_1,\overline v_2\end{matrix} &\begin{matrix}\oplus \\ \ \end{matrix}
&\begin{matrix}A_1\\ \tilde n_1\end{matrix}&\begin{matrix}\oplus \\ \ \end{matrix}
&\begin{matrix}A_1\\ \tilde n_2\end{matrix}&\begin{matrix}\oplus \\ \ \end{matrix}
&\begin{matrix}A_1\\ \tilde n_3\end{matrix}&\begin{matrix}\oplus \\ \ \end{matrix}
&\begin{matrix}A_1\\ \tilde n_6\end{matrix}
\end{tikzcd}}\]
The lattice $\widehat{\pi_{\tau}}_*H^2(\tilde Z_\varphi, \mathbb Z)$ can be obtained by $\mathbb Q$-linear extension applied to the elements $\tilde\alpha,\tilde\gamma,\tilde\varepsilon,\tilde\zeta,\tilde\eta$ which are the image via $\pi_{\varphi*}$ of the elements \eqref{generatori}, and $\nu, t_1,\dots, t_6$ defined in \eqref{generatoriZphi}. We denote $\tilde\star=\widehat{\pi_{\tau*}}\star$; if $\star=\pi_{\varphi*}\bullet$ for $\bullet\in H^2(X, \mathbb Z)$, then $\overline{\bullet}=\tilde\star$.
\end{proposition}

Again, calling $m_1,\dots m_8$ the $(-2)$-classes that generate over $\mathbb Q$ the exceptional lattice $N$ of the resolution of the singularities $\tilde Y\rightarrow \tilde Z_\varphi/\hat\tau$, we construct $H^2(\tilde Y,\mathbb Z)$ as overlattice of $\widehat{\pi_{\tau}}_*H^2(\tilde Z_\varphi, \mathbb Z)\oplus N$ using the following elements:
\begin{align*}\label{generatoriYphi}
h_1&=(\overline e_2+\overline a_1-\tilde t_4-\overline\varepsilon-\overline\zeta-\tilde n_1-\tilde t_3)/2+(m_5+m_3+m_2+m_8)/2;\\
h_2&=(\overline a_2-\tilde t_4-\overline\zeta-\tilde n_1-\tilde t_3)/2+(m_4+m_3+m_2+m_8)/2; \\
h_3&=(\overline g_2+\overline e_2+\overline e_1+\overline a_2-\tilde t_4-\overline\zeta-\tilde n_1-\tilde t_3)/2+(m_7+m_5+m_4+m_3)/2;\\
h_4&=(\overline g_2+\overline e_2+\overline a_2-\tilde t_4-\overline\zeta-\tilde n_1-\tilde t_3)/2+(m_7+m_5+m_4+m_8)/2;\\
h_5&=(\overline g_1+\overline e_2+\overline a_1-\overline\varepsilon+\overline a_2)/2+(m_6+m_5+m_4+m_2)/2;\\
h_6&=(\overline g_1+\overline g_2+\overline e_1)/2+(m_7+m_6+m_3+m_8)/2.
\end{align*}

\begin{proposition}\label{prop:M22phi}
The lattice $M_{2,2}$ is generated over $\mathbb Q$ by the elements $\tilde n_1,\tilde n_2,\tilde n_3,\tilde n_6$ $m_1,\dots, m_8$. To get a set of $\mathbb Z$-generators, add the elements
\[\mu_1=\frac{m_1+\dots+ m_8}{2},\quad\mu'_2=\frac{\tilde n_1+\tilde n_2+\tilde n_3+\tilde n_6+m_3+m_4+m_5+m_8}{2}.\]
\end{proposition}

\begin{remark}
We won't give the explicit change of basis of $H^2(\tilde Y, \mathbb Z)$ between $\tilde Y$ obtained as quotient of $\tilde Z_\tau$ or $\tilde Z_\varphi$. Notice however that the lattice $M_{2,2}$ is preserved by the change of basis: indeed, it is generated over $\mathbb Q$ by the exceptional curves introduced in the resolution of $X/(\mathbb Z/2\mathbb Z)^2$, which do not depend on the intermediate quotient.
\end{remark}

\subsection{A lattice-theoretic characterization of the intermediate quotient surface}

Let $X$ be a K3 surface with a symplectic action of $G=(\mathbb Z/2\mathbb Z)^2$, $\iota\in G$ an involution: we give a lattice-theoretic characterization of $\tilde Z$, the resolution of singularities of $X/\iota$.

\begin{theorem}\label{Ztaunonproj}
Let $\tilde Z$ be a K3 surface such that $rk(NS(\tilde Z))=12$. There exists a K3 surface $X$ with a symplectic  action of $(\mathbb Z/2\mathbb Z)^2$ such that $\tilde Z$ is birationally equivalent to the quotient $X/\iota$, where $\iota$ is one of the generators of $(\mathbb Z/2\mathbb Z)^2$, if and only if $NS(\tilde Z)=\Gamma_{2,2}$ (see Def. \ref{def:Gamma22}).
\end{theorem}
\begin{proof}
The ``only if'' is true by construction (see Sections \ref{sec:pitau}, \ref{sec:piphi}). Conversely,
suppose $NS(\tilde Z)=\Gamma_{2,2}$: the embedding $\Omega_2\subset\Gamma_{2,2}$ described in Remark \ref{Omega2inGamma22} defines a symplectic involution $\hat\rho$ on $\tilde Z$, and the Néron-Severi lattice of the resolution $\tilde Y$ of $Y=\tilde Z/\hat\rho$ is a copy of $M_{2,2}$, as proved in Proposition \ref{prop:M22}; therefore, by Nikulin's results in \cite{Nikulin2} the surface $\tilde Y$ is the resolution of the quotient of a K3 surface $X$ by the symplectic action of $(\mathbb Z/2\mathbb Z)^2$, and it holds $NS(X)=\Omega_{2,2}$. The action of $(\mathbb Z/2\mathbb Z)^2$ on $\Omega_{2,2}$ defines three copies of $\Omega_2\subset\Omega_{2,2}$, as described in Section \ref{sec:azioni}; choose one of them, and define $\iota$ as the involution for which it is the co-invariant lattice (this is always possible by the Torelli theorem). Taking the quotient map $\pi_{\iota}: X\rightarrow X/\iota$ and the resolution $\widetilde{X/\iota}$, it then holds $NS(\widetilde{X/\iota})\simeq NS(\tilde Z)$.
\end{proof}

\subsection{The dual maps}
In this section, we give a description of the maps $\pi_{\tau}^*: H^2(\tilde Z_\tau, \mathbb Z)\rightarrow H^2(X, \mathbb{Z}), \pi_{\varphi}^*: H^2(\tilde Z_\varphi, \mathbb Z)\rightarrow H^2(X, \mathbb{Z}), \pi_{2,2}^*: H^2(\tilde Y, \mathbb Z)\rightarrow H^2(X, \mathbb{Z})$; we're going to use the descriptions of $H^2(\tilde Z_\tau, \mathbb Z), H^2(\tilde Z_\varphi, \mathbb Z),H^2(\tilde Y, \mathbb Z)$ provided in Sections \ref{sec:pitau}, \ref{sec:piphi}, \ref{sec:hatrho} respectively.
The proof of the following proposition is similar to that of \cite[Prop. 4.4.1]{P}.

\begin{proposition}
\begin{enumerate}
\item The map $\pi_\tau^*: H^2(\tilde Z_\tau,\mathbb Z)\rightarrow H^2(X,\mathbb Z)$ annihilates $N$, and acts on  $\pi_{\tau*}W\subset \pi_{\tau*}\Lambda_{\mathrm{K3}}$ as follows\\[10pt]
\adjustbox{scale=0.9,center}{\xymatrix@R=0.01pc @C=0.2pc{
{\pi_\tau}^*: &A_2^{\oplus 3}  \ar@{}[r]|(.4){\oplus} &A_2(2)^{\oplus 2}  \ar@{}[r]|(.6){\oplus} &U(6) \ar@{}[r]|(.45){\oplus} &{\begin{bmatrix}
8 & 4 \\
4 & 8 \end{bmatrix}} &\ \ar[rr] &\ &{\quad A_2^{\oplus 6}} \ar@{}[r]|{\oplus} &A_2^{\oplus 2}  \ar@{}[r]|{\oplus} &U(3) \ar@{}[r]|(.45){\oplus} &{\begin{bmatrix}
4 & 2 \\
2 &4\end{bmatrix}}\\
\ &{\left(\begin{matrix}\hat a_1, \hat a_2\\ \hat c_1, \hat c_2\\ \hat g_1, \hat g_2 \end{matrix}\right.} &{\begin{matrix}\hat e_1, \hat e_2\\ \hat f_1, \hat f_2\end{matrix}} &{\hat x, \hat y} &{\hat v_1,\ \hat v_2\left.\begin{matrix}\ \\ \ \\ \ \\ \end{matrix}\right)} &\ \ar@{|->}[rr] &\ &{\left(\begin{matrix}a_1+b_1, a_2+b_2\\ c_1+d_1, c_2+d_2\\ g_1+h_1, g_2+h_2\end{matrix}\right.}  &{\begin{matrix}2e_1, 2e_2\\ 2f_1, 2f_2\end{matrix}}  &{2x, 2y} &{2v_1,2v_2\left.\begin{matrix}\ \\ \ \\ \ \\ \end{matrix}\right)}
}}

Its action can be extended to $\pi_{\tau*}\Lambda_{\mathrm{K3}}$ adding the following elements (and their respective image to the image lattice):
$\hat\alpha=(-\hat a_1+\hat a_2+\hat c_1-\hat c_2-\hat e_1+\hat e_2+\hat f_1-\hat f_2)/3,\ 
\hat \gamma=(\hat x-\hat y-\hat e_1+\hat e_2-\hat f_1+\hat f_2)/3,\ 
\hat \delta=(\hat x-2\hat c_1+2\hat c_2-\hat e_1+\hat e_2)/3,\ 
\hat \varepsilon=(\hat x+\hat c_1-\hat c_2-\hat e_1+\hat e_2)/3,\ 
\hat \zeta=(\hat x+\hat c_1+\hat c_2+\hat e_1+\hat e_2+\hat \varepsilon)/2+\hat v_2/2+\hat g_1+\hat g_2,\ 
\hat \eta=(\hat x+\hat c_1+\hat c_2+\hat e_1+\hat e_2+\hat \varepsilon)/2+(\hat g_1-\hat g_2-\hat v_2)/3+\hat v_1/6$;

to extend the action to $H^2(\tilde Z_\tau, \mathbb Z)$, add also $s_1,\dots , s_6$ as in \eqref{generatoriZtau}.

\item The map $\pi_\varphi^*: H^2(\tilde Z_\varphi,\mathbb Z)\rightarrow H^2(X,\mathbb Z)$ annihilates $N$, and acts on  $\pi_{\varphi*}W\subset \pi_{\varphi*}\Lambda_{\mathrm{K3}}$ as follows\\
\adjustbox{scale=0.9,center}{\xymatrix@R=0.01pc @C=0.2pc{
{\pi_\varphi}^*: &A_2^{\oplus 4}  \ar@{}[r]|(.4){\oplus} &A_2(4)  \ar@{}[r]|(.6){\oplus} &U(6) \ar@{}[r]|(.45){\oplus} &{\begin{bmatrix}
8 & 4 \\
4 & 8 \end{bmatrix}} &\ \ar[rr] &\ &{\quad A_2^{\oplus 6}} \ar@{}[r]|{\oplus} &A_2^{\oplus 2}  \ar@{}[r]|{\oplus} &U(3) \ar@{}[r]|(.45){\oplus} &{\begin{bmatrix}
4 & 2 \\
2 &4\end{bmatrix}}\\
\ &{\left(\begin{matrix}\tilde a_1, \tilde a_2\\ \tilde b_1, \tilde b_2\\ \tilde e_1, \tilde e_2\\ \tilde g_1, \tilde g_2 \end{matrix}\right.} & {\tilde z, \tilde w} &{\tilde x, \tilde y} &{\tilde v_1,\ \tilde v_2\left.\begin{matrix}\ \\ \ \\ \ \\ \ \\ \end{matrix}\right)} &\ \ar@{|->}[rr] &\ &{\left(\begin{matrix}a_1+d_1, a_2+d_2\\ b_1+c_1, b_2+c_2\\ e_1+f_1, e_2+f_2\\ g_1+h_1, g_2+h_2\end{matrix}\right.}  &{2z, 2w}  &{2x, 2y} &{2v_1,2v_2\left.\begin{matrix}\ \\ \ \\ \ \\ \ \\ \end{matrix}\right)}
}}

Its action can be extended to $\pi_{\varphi*}\Lambda_{\mathrm{K3}}$ adding the following elements (and their respective image to the image lattice):
$\tilde\gamma=(\tilde x-\tilde y-2\tilde e_1+2\tilde e_2)/3,\ 
\tilde\delta=(\tilde x-\tilde b_1+\tilde b_2-\tilde a_1+\tilde a_2-\tilde e_1+\tilde e_2)/3,\ 
\tilde\varepsilon=(\tilde x-\tilde z+\tilde w+\tilde b_1-\tilde b_2-\tilde e_1+\tilde e_2)/3,\ 
\tilde\zeta=(\tilde x+\tilde z+\tilde b_1+\tilde b_2+\tilde e_1+\tilde e_2+\tilde \varepsilon)/2+\tilde v_2/2+\tilde g_1+\tilde g_2,\ 
\tilde\eta=(\tilde x+\tilde b_1+\tilde b_2+\tilde e_1+\tilde e_2+\tilde \varepsilon)/2+(\tilde g_1-\tilde g_2-\tilde v_2)/3+\tilde v_1/6;$\\
to extend the action to $H^2(\tilde Z_\varphi, \mathbb Z)$, add also $t_1,\dots , t_6$ as in \eqref{generatoriZphi}.

\item The map $\pi_{2,2}^*H^2(\tilde Y,\mathbb Z)\rightarrow H^2(X,\mathbb Z)$ annihilates $M_{2,2}$, and acts on  $\pi_{2,2*}W\subset \pi_{2,2*}\Lambda_{\mathrm{K3}}$ as follows\\
\adjustbox{scale=0.9,center}{\xymatrix@R=0.02pc @C=0.2pc{
{\pi_{2,2}}^*: &A_2  \ar@{}[r]|(.4){\oplus} &A_2(2)^{\oplus 2} \ar@{}[r]|(.55){\oplus} &U(12) \ar@{}[r]|(.45){\oplus} &{\begin{bmatrix}
16 & 8 \\
8 & 16 \end{bmatrix}} &\ \ar[rr] &\ &{\quad A_2^{\oplus 8}} \ar@{}[r]|{\oplus}  &U(3) \ar@{}[r]|(.45){\oplus} &{\begin{bmatrix}
4 & 2 \\
2 &4\end{bmatrix}}\\
\ &{\Big(\overline a_1, \overline a_2} & {\begin{matrix}\overline e_1,\overline e_2\\ \overline g_1,\overline g_2\end{matrix}} &{\overline x, \overline y} &{\overline v_1,\ \overline v_2\Big)} &\ \ar@{|->}[rr] &\ &{\left(\begin{matrix}a_1+b_1+c_1+d_1\\ a_2+b_2+c_2+d_2\\ 2e_1+2f_1, 2e_2+2f_2\\ 2g_1+2h_1,2g_2+2h_2\end{matrix}\right.}   &{4x, 4y} &{4v_1,4v_2\left.\begin{matrix}\ \\ \ \\ \ \\ \end{matrix}\right)}
}}

Its action can be extended to $\pi_{\varphi*}\Lambda_{\mathrm{K3}}$ adding the following elements (and their respective image to the image lattice):
$\overline\gamma=(\overline x-\overline y-2\overline e_1+2\overline e_2)/3,\ 
\overline\varepsilon=(\overline x+\overline a_1-\overline a_2-\overline e_1+\overline e_2)/3,\ 
\overline\zeta=(\overline x+\overline a_1+\overline a_2+\overline e_1+\overline e_2+\overline v_2+\overline\varepsilon)/2+\overline g_1+\overline g_2,\ 
\overline\eta=(x+\overline a_1+\overline a_2+\overline e_1+\overline e_2+\overline\varepsilon)/2+(\overline g_1-\overline g_2-\overline v_2)/3+\overline v_1/6$;\\
to extend the action to $H^2(\tilde Y, \mathbb Z)$, add also $\overline\gamma/2$ and $k_1,\dots , k_6$ (see \eqref{generatoriY}).
\end{enumerate}
\end{proposition}

\begin{corollary}\label{pi22notprimitive}
The image of the map $\pi_{2,2}^*: H^2(\tilde Y,\mathbb Z)\rightarrow H^2(X,\mathbb Z)$ is not primitive in $H^2(X,\mathbb Z)$: indeed, it is a sublattice of index $2^3$ of the invariant lattice for the action of $(\mathbb Z/2\mathbb Z)^2$ on $X$.
\end{corollary}
\begin{proof}
The elements $
\pi_{2,2}^*(k_4+k_6)/2,
\pi_{2,2}^*(\overline a_1+\overline\varepsilon+\overline\eta)/2,
\pi_{2,2}^*\overline\gamma/4$ are integral in  $H^2(X,\mathbb Z)$, but they do not belong to $\pi_{2,2}^*H^2(\tilde Y,\mathbb Z)$: indeed, a $\mathbb Z$-basis of the latter is given by the image via $\pi_{2,2}^*$ of $\{\overline a_1,\overline\gamma/2,\overline\varepsilon,\overline\eta, k_1,\dots,k_6\}$.
\end{proof}

\section{Projective families of K3 surfaces with a symplectic action \\of $(\mathbb Z/2\mathbb Z)^2$ and their quotients}

It was already known by Nikulin that the correspondence between surfaces $X$ that admit a symplectic action of an abelian group $G$, and surfaces $\tilde Y$ that are the resolution of singularities of $X/G$, is actually a moduli spaces correspondence \cite[Prop. 2.9]{Nikulin2}: $X$ is characterized by the existence of a primitive embedding $\Omega_G\hookrightarrow NS(X)$, which is an equality in the most general case, and similarly $\tilde Y$, with the lattice $M_G$ instead of $\Omega_G$.\\ In the projective case the moduli spaces split in irreducible components, that we'll refer to as \emph{projective families}, classified by the Néron-Severi lattice of their general member: this will always be for $X$ a cyclic overlattice of $\Omega_G\oplus\langle 2d\rangle$, with $d>0$ (see \cite[Prop. 2.9]{Nikulin2}, \cite[Prop. 2.2]{VGS}), for $\tilde Y$ a cyclic overlattice of $M_G\oplus\langle 2d\rangle$. The correspondence between families of $X$ and $\tilde Y$ has been completely described in \cite{VGS},\cite{GS} for symplectic involutions, and in \cite{GP} and \cite{P} for symplectic automorphisms of order 3 and 4. We remark that, while there is always a bijection between families of $X$ and $\tilde Y$, the same does not hold for the intermediate quotient surface $\tilde Z$ that appears when $G$ has order 4.

\begin{remark}\label{overlatticenotation}
\textbf{Notation.} Consider the lattice $S\oplus \langle k \rangle$, where $S$ is a negative definite even lattice and $\langle k \rangle$ is an even positive definite lattice with intersection matrix $[k]$.\\
Denote $(S\oplus \langle k \rangle)'$  any overlattice of index 2 of $S\oplus \langle k \rangle$ obtained by adding to the list of generators a class of the form $(s+\kappa)/2$, with $s\in S$ and $\kappa$ the generator of $\langle k \rangle$. 
When two such overlattices are not isomorphic as abstract lattices, they will be denoted as $(S\oplus\langle k\rangle)'^{(i)},\ i=1,2$. Similarly, denote $(S\oplus \langle k \rangle)^{\star}$ any overlattice of index 4 obtained by a class of the form $(s+\kappa)/4$.
\end{remark}

\begin{lemma}\label{Lisample}
Let $X$ be a projective K3 surface with a symplectic action of $G$, such that $NS(X)$ has signature $(1,rk(\Omega_G))$. Then we may assume that $L=\Omega_G^{\perp_{NS(X)}}$ is ample.
\end{lemma}
\begin{proof}
We may assume that $L$ is effective up to a sign change, because $L^2 > 0$. Then, since there are no $(-2)$-classes in $L^\perp = \Omega_G$ \cite[Thm. 4.3]{Nikulin2}, any $(-2)$-curve has class of the form $nL + w$ with $n\in\mathbb N$ and $w\in\Omega_G$: classes of this form have positive intersection with $L$, so $L$ is ample by the Nakai-Moishezon criterion.
\end{proof}

\subsection[K3 surfaces with a symplectic action of $(\mathbb Z/2\mathbb Z)^2$]{Projective families of K3 surfaces with a symplectic action of $(\mathbb Z/2\mathbb Z)^2$}

In this section, we're going to classify the cyclic overlattices of $\Omega_{2,2}\oplus\langle 2d\rangle$ that admit a primitive embedding in $\Lambda_{\mathrm{K3}}$, and are therefore admissible as Néron-Severi lattice of the general member of a projective family of K3 surfaces with a symplectic action of $(\mathbb Z/2\mathbb Z)^2$.

\begin{definition}\label{reldeq}
Consider an even lattice $S$, its group of isometries $O(S)$ and its discriminant group $A_S$ with discriminant form $q_S$. We define on $A_S$ the equivalence relation $\approx_S$: two elements $r,s \in A_S$ are in relation if there exists an isometry $\overline\varphi\in O(A_S)$ induced by an isometry $\varphi\in O(S)$ such that $\overline\varphi(r)=s$; we will denote the equivalence classes for this relation with the triple $(k,g,n)$, where $k$ is the order of the subgroup $\langle r\rangle\subset A_S$, $g=q_S(r)\in\mathbb Q/2\mathbb Z$ is the square of the generator, and $n$ is the cardinality of the class. In our case, this triple is sufficient to uniquely identify each class.
\end{definition}

\begin{proposition}\label{classiOmega22}
The equivalence classes for $\approx_{\Omega_{2,2}}$ are given in the table below: for each one we give a representative element $x_{(k,g,n)}$ in terms of the generators of $\Omega_{2,2}\subset H^2(X,\mathbb Z)$ as described in Section \ref{sec:azioni}.
\end{proposition}

\begin{longtable}{|c|c|}
\hline
class $(k,g,n)$ & representative $x_{(k,g,n)}$\TBstrut\\
\hline
 $(2,0,108)$ & $({f_1-e_1+h_1-g_1})/{2}$\TBstrut\\
\hline
 $(2,0,3)$ & $({b_1+c_1+d_1-3a_1})/{2}$\TBstrut\\
\hline
 $(2,1,108)$ & $({w+f_1-e_1+h_1-g_1})/{2}$\TBstrut\\
\hline
 $(2,1,36)$ &$w/2$\TBstrut\\ 
\hline
 $(4,1/2,384)$ &$({b_1+c_1+d_1-3a_1})/{4}$\TBstrut\\
\hline
 $(4,3/2,384)$ &$({b_2+c_2+d_2-3a_2})/{4}+{w}/{2}$\TBstrut\\
\hline
\end{longtable}

\begin{theorem}\label{NSpossibiliKlein}
Let $X$ be a projective K3 surface that admits a symplectic action of $(\mathbb Z/2\mathbb Z)^2$, such that $rk(NS(X))=13$. Then, using the notation in Remark \ref{overlatticenotation}, $NS(X)$ is one of the following lattices:
\begin{enumerate}
\item for every $d\in \mathbb N$, $NS(X)=\Omega_{2,2}\oplus\langle 2d \rangle$;
\item for any $d=_4 0$ there are two non-isometric possibilities: $NS(X)=(\Omega_{2,2}\oplus\langle 2d \rangle)'^{(i)}$, $i=1,2$; 
\item for $d=_4 2$, $NS(X)=(\Omega_{2,2}\oplus\langle 2d \rangle)'$: this lattice is uniquely determined by $d$ and the index, but it admits two non isomorphic embeddings $\iota_1,\iota_2:\Omega_{2,2}\hookrightarrow NS(X)$, i.e. no isometry $\psi\in O(\Omega_{2,2})$ exists such that $\iota_1=\iota_2\circ\psi$;
\item For $d=_{16} 4$ or $d=_{16} -4$, $NS(X)=(\Omega_{2,2}\oplus\langle 2d \rangle)^{\star}$, uniquely determined by $d$ and the index of the overlattice.
\end{enumerate}
Each of these lattices admits a unique primitive embedding in $\Lambda_{\mathrm{K3}}$ up to isometries of the latter.
\end{theorem}
\begin{proof}
By \cite[Prop. 1.4.1.a]{Nikulin1}, overlattices of index $k$ of $\Omega_{2,2}\oplus\langle 2d \rangle$ correspond to isotropic elements in $A_{\Omega_{2,2}\oplus\langle 2d \rangle}$ of the form $(L+v)/k$, where $L$ generates $\langle 2d \rangle$ and $v\in\Omega_{2,2}$ is chosen up to the action of $O(\Omega_{2,2})$ on $A_{\Omega_{2,2}}$. Requiring $(L+v)/2$ to be isotropic, we see that for each value of $d$ modulo $4$, $v/2$ belongs to one of the classes of $\approx_{\Omega_{2,2}}$ containing elements of order 2. We then check if the corresponding overlattices are isometric as lattices or not, by comparing their discriminant forms. A similar argument applies for overlattices of index 4. The uniqueness of the primitive embedding of each admissible $NS(X)$ in $\Lambda_{\mathrm{K3}}$ follows by \cite[Prop. 1.14.1]{Nikulin1}.
\end{proof}

In Table \ref{esempiLKlein} we exhibit a primitive embedding in $\Lambda_{\mathrm{K3}}$ of each of the lattices presented in Theorem \ref{NSpossibiliKlein}, following the same process as \cite[Ex. 5.1.6]{P}: having fixed the primitive embedding of $\Omega_{2,2}$ in $H^2(X,\mathbb Z)\simeq\Lambda_{\mathrm{K3}}$ as described in Section \ref{sec:azioni}, we provide examples of primitive classes  $L\in \Omega_{2,2}^{\perp_{\Lambda_{\mathrm{K3}}}}$ such that $L^2=2d$ and $L$ glues to one of the elements $x_{(k,g,n)}$ in Proposition \ref{classiOmega22}, i.e. $(L+x_{(k,g,n)})/k$ is integral in $\Lambda_{\mathrm{K3}}$. Then we define $NS(X)$ as the primitive saturation of $\Omega_{2,2}\oplus\mathbb ZL$. We remark that we may assume that $L$ is ample by Lemma \ref{Lisample}, and that using $x_{(2,1,108)}$ and $x_{(2,1,36)}$ we obtain isomorphic Néron-Severi lattices.

\begin{table}[h!]\caption{Examples of ample classes} \label{esempiLKlein}  
\footnotesize
\centering
\begin{tabular}{|c|c|c|c|}
\hline
$d$ &{$NS(X)$\Tstrut} & $L$ & $(k,g,n)$\\ [4pt]
\hline
any $d$\Tstrut &$\Omega_{2,2}\oplus\langle 2d\rangle$\Tstrut &$L_0(d):=(x+2y-e_1-f_1+e_2+f_2)/3+dy$ &\  \\ [4pt]
\hline
{$d=4h$\Tstrut} & $(\Omega_{2,2}\oplus\langle 2d\rangle)'^{(1)}$\Tstrut & $L_{2,0}^{(1)}(h):=2L_0(h)+e_1+f_1+g_1+h_1$ &$(2,0,108)$\Tstrut\\ [3pt]
{$d=4(h-1)$\Tstrut} & $(\Omega_{2,2}\oplus\langle 2d\rangle)'^{(2)}$\Tstrut & $L_{2,0}^{(2)}(h):=2L_0(h)+a_1+b_1+c_1+d_1$ &$(2,0,3)$\Tstrut\\ [4pt]
\hline
\multirow{2}{*}{$d=4h+2$\Tstrut} & \multirow{2}{*}{$(\Omega_{2,2}\oplus\langle 2d\rangle)'$\Tstrut} & $L_{2,2}^{(a)}(h):=2L_0(h)+v_2+f_1+e_1+h_1+g_1$ &$(2,1,108)$\Tstrut\\ [3pt]
\ & \ &$L_{2,2}^{(b)}(h):=2L_0(h)+v_2$\Tstrut &$(2,1,36)$\Tstrut\\ [4pt]
\hline
{$d=16h-4$\Tstrut} & {$(\Omega_{2,2}\oplus\langle 2d\rangle)^\star$\Tstrut} & $L_{4,-4}(h):=4L_0(h)+a_1+b_1+c_1+d_1$ &$(4,1/2,384)$\Tstrut\\ [3pt]
\hline
{$d=16h+4$\Tstrut} &{$(\Omega_{2,2}\oplus\langle 2d\rangle)^\star$\Tstrut} & $L_{4,4}(h):=4L_0(h)+2v_2+a_2+b_2+c_2+d_2$ &$(4,3/2,384)$\Tstrut\\ [4pt]
\hline
\end{tabular}\normalsize
\end{table}

\subsection[K3 surfaces that are resolution of $X/(\mathbb Z/2\mathbb Z)^2$]{Projective families of K3 surfaces that arise as resolution of the singularities of $X/(\mathbb Z/2\mathbb Z)^2$}

\begin{proposition}\label{classiM22}
We give a representative element $x_{(k,g,n)}$ for each non-trivial equivalence class $(k,g,n)$ for $\approx_{M_{2,2}}$ (see Definition \ref{reldeq}) in terms of the generators of $M_{2,2}$ introduced in Proposition \ref{prop:M22}.
\end{proposition}

\begin{longtable}{|c|c|}
\hline
class $(k,g,n)$ & representative $x_{(k,g,n)}$\TBstrut\\
\hline
 $(2,0,54)$ & $(\overline n_1+\overline n_4+m_1+m_7)/2$\TBstrut\\
\hline
 $(2,0,1)$ & $(m_3+m_4+m_5+m_6)/2$\TBstrut\\
\hline
 $(2,1/2,64)$ & $(\overline n_1+m_1+m_6)/2$\TBstrut\\
\hline
 $(2,1,54)$ & $(\overline n_1+\overline n_4+m_1+m_4+m_5+m_7)/2$\TBstrut\\
\hline
 $(2,1,18)$ & $(m_3+m_6)/2$\TBstrut\\
\hline
 $(2,3/2,64)$ &$(\overline n_1+m_1+m_3+m_4+m_5)/2$\TBstrut\\ 
\hline
\end{longtable}

\begin{theorem}\label{NSpossibiliY}
Let $\tilde Y$ be a projective K3 surface such that $rk(NS(\tilde Y))=13$ and $NS(\tilde Y)$ contains primitively $M_{2,2}$ and $\langle 2e \rangle,\ e\in\mathbb N\setminus\{0\}$. Then, $NS(\tilde Y)$ is one of the following:
\begin{enumerate}
\item for every $e$, $NS(\tilde Y)=M_{2,2}\oplus\langle 2e \rangle$;
\item for every $e$, $NS(\tilde Y)$ is an overlattice of index 2 of $M_{2,2}\oplus\langle 2e \rangle$. If $e=_4 0$ there are two non isomorphic possibilities for $NS(\tilde Y)$: $(M_{2,2}\oplus\langle 2e \rangle)'^{(i)}$, $i=1,2$; if $e=_4 2$ $NS(\tilde Y)=(M_{2,2}\oplus\langle 2e \rangle)'$ is unique, but there are two non isomorphic embeddings of $M_{2,2}$ in $NS(\tilde Y)$. If $e$ is odd, this overlattice uniquely determined by $e$ and the index of the overlattice.
\end{enumerate}
Each of these lattices admits a unique primitive embedding in $\Lambda_{\mathrm{K3}}$.
\end{theorem}
\begin{proof}
The overlattices of $M_{2,2}\oplus\langle 2e \rangle$ are in bijection with the equivalence classes for $\approx_{M_{2,2}}$.
Fix the primitive embedding $M_{2,2}\hookrightarrow\Lambda_{\mathrm{K3}}$ as in Proposition \ref{prop:M22}:
the orthogonal complement of $M_{2,2}$ is the overlattice of index 2 of the lattice $(\pi_{2,2})_*H^2(X, \mathbb Z)$ obtained by the addition of $\overline\gamma/2$ as generator. We can therefore use as generators of the lattice $\langle 2e\rangle=M_{2,2}^{\perp_{NS(\tilde Y)}}$ one of the primitive classes $\overline L$ in $H^2(\tilde Y, \mathbb Z)$ obtained from $(\pi_{2,2})_*L$ (with $L$ one of those in Table \ref{esempiLKlein}): all the equivalence classes for the relation $\approx_{M_{2,2}}$ have a representative (not necessarily $x_{(k,g,n)}$) that glues to one of the $\overline L$. We then check the uniqueness of the primitive embedding of the resulting $NS(\tilde Y)$ in $\Lambda_{\mathrm{K3}}$ using again \cite[Prop. 1.14.1]{Nikulin1}.
\end{proof}

\begin{theorem}\label{relations22}
In the following table we give the correspondence between families of K3 surfaces $X$ with a symplectic action of $(\mathbb Z/2\mathbb Z)^2$, and $\tilde Y$ which is the minimal resolution of the quotient  $X/(\mathbb Z/2\mathbb Z)^2$, with the notation of Remark \ref{overlatticenotation}. 
The primitive classes $\overline L\in NS(\tilde Y)$ that generate the sublattices $\langle nd\rangle$ as stated are indicated in curly brackets.
\begin{center}
\footnotesize
\begin{tabular}{|c|c|c c|}
\hline
\multicolumn{2}{|c|}{$NS(X)$\Tstrut}    		    &\multicolumn{2}{c|} {$NS(\tilde Y)$\Tstrut} \\ [4pt]
\hline
$d=_2 1$\Tstrut	& $\Omega_{2,2}\oplus\langle 2d\rangle$\Tstrut & $(M_{2,2}\oplus\langle 2d\rangle)'$ &$\{\overline L=\pi_{2,2*}L_0/2\}$\Tstrut\\ [4pt]
\hline
\multirow{2}{*}{$d=_4 2$\Tstrut} & $\Omega_{2,2}\oplus\langle 2d\rangle$\Tstrut & $(M_{2,2}\oplus\langle 8d\rangle)'$ &$\{\overline L=\pi_{2,2*}L_0\}$\Tstrut\\ [3pt]
\ & $(\Omega_{2,2}\oplus\langle 2d\rangle)'$\Tstrut &$(M_{2,2}\oplus\langle 2d\rangle)'$ & $ \big\{\overline L=\frac{\pi_{2,2*}L^{(k)}_{2,2}}{2}, k=a,b\big\} $\Tstrut \\[4pt]
\hline
\multirow{4}{*}{$d=_8 0$\Tstrut} & $\Omega_{2,2}\oplus\langle 2d\rangle$\Tstrut & $(M_{2,2}\oplus\langle 8d\rangle)'$ &$\{\overline L=\pi_{2,2*}L_0\}$\Tstrut\\ [3pt]
\ & $(\Omega_{2,2}\oplus\langle 2d\rangle)'^{(1)}$\Tstrut 
&$(M_{2,2}\oplus\langle 2d\rangle)'$ & $ \big\{\overline L=\frac{\pi_{2,2*}L^{(1)}_{2,0}}{2}\big\} $\Tstrut \\[3pt]
\ & $(\Omega_{2,2}\oplus\langle 2d\rangle)'^{(2)}$\Tstrut &$M_{2,2}\oplus\langle d/2\rangle $ & $ \big\{\overline L=\frac{\pi_{2,2*}L^{(2)}_{2,0}}{4}\big\} $\Tstrut \\[4pt]
\hline
\multirow{5}{*}{$d=_8 4$\Tstrut} & $\Omega_{2,2}\oplus\langle 2d\rangle$\Tstrut & $(M_{2,2}\oplus\langle 8d\rangle)'$ &$\{\overline L=\pi_{2,2*}L_0\}$\Tstrut\\ [3pt]
\ & $(\Omega_{2,2}\oplus\langle 2d\rangle)'^{(1)}$\Tstrut 
&$(M_{2,2}\oplus\langle 2d\rangle)'^{(1)}$ & $ \big\{\overline L=\frac{\pi_{2,2*}L^{(1)}_{2,0}}{2}\big\} $\Tstrut \\[3pt]
\ & $(\Omega_{2,2}\oplus\langle 2d\rangle)'^{(2)}$\Tstrut &$(M_{2,2}\oplus\langle 2d\rangle)'^{(2)} $\Tstrut & $ \big\{\overline L=\frac{\pi_{2,2*}L^{(2)}_{2,0}}{2}\big\} $\Tstrut \\[3pt]
\ & $(\Omega_{2,2}\oplus\langle 2d\rangle)^\star$\Tstrut & $M_{2,2}\oplus\langle d/2\rangle$ & $ \big\{\overline L=\frac{\pi_{2,2*}L_{4,-4}}{4},\frac{\pi_{2,2*}L_{4,4}}{4}\big\}$\Tstrut \\ [6pt]
\hline
\end{tabular}
\end{center}\normalsize
\end{theorem}

\subsection[K3 surfaces that are intermediate quotient for the symplectic action of $(\mathbb Z/2\mathbb Z)^2$]{Projective families of K3 surfaces that are intermediate quotient for the symplectic action of $(\mathbb Z/2\mathbb Z)^2$}\label{proj_intermediate}

The general projective surface $\tilde Z$ that is the resolution of singularities of $X/\langle\iota\rangle$, where $\iota$ is one of the involutions in $(\mathbb Z/2\mathbb Z)^2$, is polarized with the lattice $\Gamma_{2,2}\oplus\langle 2x \rangle,\ x>0$ or one of its cyclic overlattices.

The lattice $\Gamma_{2,2}$ is described in Definition \ref{def:Gamma22}; since its orthogonal complement in $H^2(\tilde Z, \mathbb Z)$ is $\pi_{\iota*}H^2(X,\mathbb Z)$, one may expect that any cyclic overlattice of $\Gamma_{2,2}\oplus\langle 2x\rangle$ be realized with $\Gamma_{2,2}^{\perp}=\pi_{\iota*}L$ (divided by an appropriate integer if not already primitive) for any $\iota$, and $L$ chosen among those in Table \ref{esempiLKlein}: however, consider the following Remark \ref{rem:different_action}.

\begin{remark}\label{rem:different_action}
The quotient maps induced by the three involutions $\tau,\varphi,\rho$ act on the classes $(2,0,108)$ and $(2,1,36)$ for $\approx_{\Omega_{2,2}}$ by killing some elements, and not others: for instance, consider $w/2$ and $(e_1-f_1)/2$, both belonging to the class $(2,1,36)$; $\pi_{\tau*}w/2=0=\pi_{\rho*}w/2$, while $\pi_{\varphi*}w/2\neq 0$; similarly, $\pi_{\varphi*}(e_1-f_1)/2=0=\pi_{\rho*}(e_1-f_1)/2$, while $\pi_{\tau*}(e_1-f_1)/2\neq 0$. We remark that with our choice of representatives $x_{(k,g,n)}$, $L_{2,0}^{(1)}$ and $L_{2,2}^{(b)}$ glue to elements on which $\pi_{\tau*}$ and $\pi_{\rho*}$ act the same, while $\pi_{\varphi*}$ acts differently. Therefore, to the end of describing the correspondence of projective families of $X,\tilde Z$ and $\tilde Y$, it is necessary and sufficient that we consider the maps in diagram \eqref{diagrammaquozienti}.
\end{remark}

\begin{remark}
The lattices $S_\tau, S_\varphi$ associated to the classes $\pi_{\tau*}L_{2,0}^{(1)}$ and  $\pi_{\varphi*}L_{2,0}^{(1)}$ are not isomorphic, and there is no other $L\neq L_{2,0}^{(1)}$ in Table \ref{esempiLKlein} such that $\pi_{\tau*}L$ and  $\pi_{\varphi*}L$ realize $S_\varphi$ and $S_\tau$ respectively. The same holds with $L_{2,2}^{(b)}$ in place of $L_{2,0}^{(1)}$.
\end{remark}

\begin{theorem}\label{NSpossibiliZKlein}
Let $\tilde Z$ be a K3 surface such that $rk(NS(\tilde Z))=13$; suppose $NS(\tilde Z)$ admits a primitive embedding of both $\Gamma_{2,2}$ and a class of positive square $2x$ that generates $\Gamma_{2,2}^{\perp_{NS(\tilde Z)}}$. Then $NS(\tilde Z)$ is one of the following: 
\begin{enumerate}
\item for any $x$, $\Gamma_{2,2}\oplus\langle2x\rangle$;
\item for any $x=_4 0$ there are two non-isomorphic overlattices of index 2: $(\Gamma_{2,2}\oplus\langle2x\rangle)'^{(i)}$, $i=1,2$;
\item for $x=_4 2$, $(\Gamma_{2,2}\oplus\langle2x\rangle)'$, uniquely determined by $x$ and the index of the overlattice;
\item for $x=_8 4$, $(\Gamma_{2,2}\oplus\langle2x\rangle)^\star$,  uniquely determined by $x$ and the index of the overlattice.
\end{enumerate}
\end{theorem}
\begin{proof}
An element of the form $(E+\alpha)/2$, with $E^2=2x$ and $\alpha\in\Gamma_{2,2}$, has integer, even self-intersection only if $x$ is even, and an element of the form $(E+\alpha)/4$ only if $x=_8 4$. \\
The equivalence classes for $\approx_{\Gamma_{2,2}}$ are presented in the following table:\begin{longtable}{|c|c|c|}
\hline
class $(k,g,n)$ & repr. $x_{(k,g,n)}\in\pi_{\tau*}H^2(X,\mathbb Z)$ & repr. $y_{(k,g,n)}\in\pi_{\varphi*}H^2(X,\mathbb Z)$  \TBstrut\\
\hline
$(2,0,3)$ & $\frac{\hat f_1-\hat e_1}{2}$ & $\frac{n_2+n_3+n_4+n_8}{2}$\TBstrut\\ 
\hline
$(2,0,8)$ & $\frac{n_3+n_5+n_6+n_8}{2}$ & $\frac{n_3+n_4+n_5+n_7}{2}$  \TBstrut\\
\hline
$(2,0,12)$ & $\frac{\hat f_1-\hat e_1+n_4+n_6+\hat c_1-\hat a_1}{2}$ & $\frac{\tilde b_1-\tilde a_1+n_6+n_7}{2}$ \TBstrut\\ 
\hline
$(2,1,4)$ & $\frac{\hat f_1-\hat e_1+n_4+n_6}{2}$ & $\frac{n_6+n_7}{2}$ \TBstrut\\ 
\hline
$(2,1,12)$  &$\frac{\hat c_1-\hat a_1}{2}$ &$\frac{\tilde b_1-\tilde a_1}{2}$\TBstrut\\ 
\hline
$(2,1,24)$ & $\frac{n_3+n_6+n_5+n_8+\hat c_1-\hat a_1}{2}$ & $\frac{\tilde b_1-\tilde a_1+n_3+n_4+n_5+n_7}{2}$ \TBstrut\\ 
\hline
$(4,1/2,96)$ & $\frac{3(\hat f_1-\hat e_1)}{4}+\frac{n_6+n_8}{2}$ &$\frac{n_2+n_3+3n_4+3n_8}{4}+\frac{x'_1}{2}$\TBstrut\\ 
\hline
$(4,3/2,96)$ & $\frac{\hat f_1-\hat e_1}{4}+\frac{n_3+n_4+n_5+n_6}{2}$ &$\frac{n_2+3n_3+n_4+3n_8}{4}+\frac{x'_1+n_5+n_6}{2}$\TBstrut\\ 
\hline
\end{longtable}
The corresponding overlattice of $\Gamma_{2,2}\oplus\langle 2x\rangle$ can be realized having fixed either the embedding $\Gamma_{2,2}\simeq\Gamma_\tau\subset H^2(\tilde Z_{\tau},\mathbb Z)$ as in Definition \ref{def:Gammatau}, or $\Gamma_{2,2}\simeq\Gamma_\varphi\subset H^2(\tilde Z_{\varphi},\mathbb Z)$ as in Lemma \ref{prop:Gamma_22phi} (these are equivalent up to isometries of $\Lambda_{\mathrm{K3}}$); then, proceed as in the proof of Theorem \ref{NSpossibiliY}.
\end{proof}

\begin{theorem}\label{correspondence:GammaKlein}
We give the correspondence between families of projective K3 surfaces $X$ with a symplectic action of $(\mathbb Z/2\mathbb Z)^2$, $\tilde Z$ the resolution of the singularities of the quotient $X/\iota$, with $\iota$ any of the generators of $(\mathbb Z/2\mathbb Z)^2$, and $\tilde Y$ as in Theorem \ref{relations22}. The notation is explained in Remark \ref{overlatticenotation}, and for $NS(\tilde Z)$ the class generating $\langle nd\rangle$ is indicated in curly brackets.
\end{theorem}
\begin{center}
\footnotesize
\begin{tabular}{|c|c|c c|c|}
\hline
\multicolumn{2}{|c|}{$NS(X)$\Tstrut}    		    &\multicolumn{2}{c|} {$NS(\tilde Z)$\Tstrut} & {$NS(\tilde Y)$\Tstrut}\\ [4pt]
\hline
$d=_2 1$\Tstrut	& $\Omega_{2,2}\oplus\langle 2d\rangle$\Tstrut & $(\Gamma_{2,2}\oplus\langle 4d\rangle)'$ &$\{\pi_{\iota*}L_0\}$\Tstrut & $(M_{2,2}\oplus\langle 2d\rangle)'$ \Tstrut\\ [4pt]
\hline
\multirow{5}{*}{$d=_4 2$\Tstrut} & $\Omega_{2,2}\oplus\langle 2d\rangle$\Tstrut & $(\Gamma_{2,2}\oplus\langle 4d\rangle)'^{(1)}$ &$\{\pi_{\iota*}L_0\}$\Tstrut & $(M_{2,2}\oplus\langle 8d\rangle)'$ \\ [3pt]
\cdashline{2-5}
\ & \multirow{2}{*}{$(\Omega_{2,2}\oplus\langle 2d\rangle)'$\TTstrut} 
&$(\Gamma_{2,2}\oplus\langle 4d\rangle)'^{(2)}$ & $ \big\{\pi_{\iota*}L^{(a)}_{2,2}\big\}$\Tstrut  &\multirow{2}{*}{$(M_{2,2}\oplus\langle 2d\rangle)'$\TTstrut} \\
\ &\ &$\Gamma_{2,2}\oplus\langle d\rangle$ & $ \big\{{\pi_{\tau*}L^{(b)}_{2,2}}/{2}\big\} $\Tstrut &\  \\
\ & \ &$(\Gamma_{2,2}\oplus\langle 4d\rangle)^\star$ & $ \big\{\pi_{\varphi*}L^{(b)}_{2,2}\big\} $\Tstrut &\  \\[4pt]
\hline
\multirow{6}{*}{$d=_8 0$\Tstrut} & $\Omega_{2,2}\oplus\langle 2d\rangle$\Tstrut &  $(\Gamma_{2,2}\oplus\langle 4d\rangle)'^{(1)}$ &$\{\pi_{\iota*}L_0\}$\Tstrut & $(M_{2,2}\oplus\langle 8d\rangle)'$ \\ [3pt]
\cdashline{2-5} 
\ & \multirow{2}{*}{$(\Omega_{2,2}\oplus\langle 2d\rangle)'^{(1)}$\Tstrut} &$\Gamma_{2,2}\oplus\langle d\rangle$ & $ \big\{{\pi_{\varphi*}L^{(1)}_{2,0}}/{2}\big\} $\Tstrut & \multirow{2}{*}{$(M_{2,2}\oplus\langle 2d\rangle)'$\Tstrut} \\
\ & \ &$(\Gamma_{2,2}\oplus\langle 4d\rangle)'^{(2)}$ & $ \big\{\pi_{\tau*}L^{(1)}_{2,0}\big\} $\Tstrut &\ \\[3pt]
\cdashline{2-5}
\ & $(\Omega_{2,2}\oplus\langle 2d\rangle)'^{(2)}$\Tstrut &$(\Gamma_{2,2}\oplus\langle d\rangle)' $ & $ \big\{{\pi_{\iota*}L^{(2)}_{2,0}}/{2}\big\} $\Tstrut &$M_{2,2}\oplus\langle d/2\rangle $ \\[4pt]
\hline
\multirow{7}{*}{$d=_8 4$\Tstrut} & $\Omega_{2,2}\oplus\langle 2d\rangle$\Tstrut &  $(\Gamma_{2,2}\oplus\langle 4d\rangle)'^{(1)}$ &$\{\pi_{\iota*}L_0\}$\Tstrut &$(M_{2,2}\oplus\langle 8d\rangle)'$ \\ [3pt]
\cdashline{2-5} 
\ & \multirow{2}{*}{$(\Omega_{2,2}\oplus\langle 2d\rangle)'^{(1)}$\Tstrut} &$\Gamma_{2,2}\oplus\langle d\rangle$ & $ \big\{{\pi_{\varphi*}L^{(1)}_{2,0}}/{2}\big\} $\Tstrut &\multirow{2}{*}{$(M_{2,2}\oplus\langle 2d\rangle)'^{(1)}$\Tstrut}  \\
\ & \ &$(\Gamma_{2,2}\oplus\langle 4d\rangle)'^{(2)}$ & $ \big\{\pi_{\tau*}L^{(1)}_{2,0}\big\} $\Tstrut &\  \\[3pt]
\cdashline{2-5}
\ & $(\Omega_{2,2}\oplus\langle 2d\rangle)'^{(2)}$\Tstrut &\multirow{2}{*}{$(\Gamma_{2,2}\oplus\langle d\rangle)' $} & $ \big\{{\pi_{\iota*}L^{(2)}_{2,0}}/{2}\big\} $\Tstrut &$(M_{2,2}\oplus\langle 2d\rangle)'^{(2)} $\Tstrut  \\
\ & $(\Omega_{2,2}\oplus\langle 2d\rangle)^\star$\Tstrut & \ & $ \big\{\frac{\pi_{\iota*}L_{4,-4}}{2},\frac{\pi_{\iota*}L_{4,4}}{2}\big\}$\Tstrut & $M_{2,2}\oplus\langle d/2\rangle$ \\ [4pt]
\hline
\end{tabular}\normalsize
\end{center}

\section{Projective models}\label{sec:projmod}

Given a nef and big divisor $L$ on $X$, there is a natural map $\phi_{|L|}: X\rightarrow\mathbb P(H^0(X,L)^*)\simeq \mathbb P^n$,
with $n=L^2/2+1$. Any automorphism $\sigma$ of $X$ that preserves $L$ induces an action on $\mathbb P(H^0(X,L)^*)$: in particular, if $\sigma$ is finite of order $m$, we can split $H^0(X,L)$ in eigenspaces corresponding to the $m$-roots of unity. 

\begin{remark}\label{rem:bigger_action} In our case, the action of $\tau$ on $H^0(X,L)$ could actually have order $2k$ for some integer $k>1$, being such that
\[\tau^2:(x_0,\dots,x_n)\mapsto\xi_k(x_0,\dots,x_n)\]
for $\xi_k$ a root of unity (and similarly $\varphi$). However, if instead of $\tau$ and $\varphi$ we consider the action of $\tau^k$ and $\varphi^h$ on $H^0(X,L)$, we don't have any control on the order of their composition $\rho$ -- we only know that it divides $kh$. Therefore, the group $G$ acting on $H^0(X,L)$ may be bigger than $(\mathbb Z/2\mathbb Z)^2$; if so, we can only conclude that $G$ is dihedral, as $\tau^k$ and $\varphi^h$ are involutions.
\end{remark}

\subsection{Eigenspaces of $\tau,\varphi$}\label{eigenspaces:tau,phi}

Let $X$ be a K3 surface with a symplectic action of $(\mathbb Z/2\mathbb Z)^2=\langle\tau,\varphi\rangle$, let $L$ be the ample class that generates $\Omega_{2,2}^{\perp_{NS(X)}}$. If $\tilde Z_\tau,\tilde Z_\varphi$ are the minimal resolution of $X/\tau,X/\varphi$ respectively, we have
\begin{align*}H^0(X,L)
&=\pi_\tau^*H^0(\tilde Z_\tau, E_1)\oplus \pi_\tau^*H^0(\tilde Z_\tau, E_2)\\ &=\pi_\varphi^*H^0(\tilde Z_\varphi, F_1)\oplus \pi_\varphi^*H^0(\tilde Z_\varphi, F_2);\end{align*}
the nef divisors $E_1, E_2\in NS(\tilde Z_\tau),F_1, F_2\in NS(\tilde Z_\varphi)$ that satisfy these equalities for the choices of ample classes introduced in Table \ref{esempiLKlein} are defined in the following tables, with the exceptional curves numbered as in Sections \ref{sec:pitau}, \ref{sec:piphi}; for the general symplectic involution on a K3 surface, this is done in \cite[Prop. 2.7]{VGS}.

\footnotesize
\begin{tabular}{|m{0.08\textwidth}<{\centering}|m{0.41\textwidth}<{\centering}|m{0.4364\textwidth}<{\centering}|}
\hline
$L_0(d) \TBstrut$ &{ $d=_2 0$ \TBstrut }                      &{ $d=_2 1$  \TBstrut} \\  \hline
$E_1\Tstrut $ & $\pi_{\tau*}L_0/2-(n_3+n_5+n_6+n_8)/2 \Tstrut$ & $\pi_{\tau*}L_0/2-(n_1+n_8)/2$\Tstrut  \\
$E_2\TBstrut $ & $\pi_{\tau*}L_0/2-(n_1+n_2+n_4+n_7)/2\TBstrut$  & $\pi_{\tau*}L_0/2-(n_2+n_3+n_4+n_5+n_6+n_7)/2$\TBstrut  \\
 \hline
$F_1\Tstrut $ & $\pi_{\varphi*}L_0/2-(n_1+n_2+n_6+n_8)/2 \Tstrut$ & $\pi_{\varphi*}L_0/2-(n_6+n_7)/2$\Tstrut  \\
$F_2\TBstrut $ & $\pi_{\varphi*}L_0/2-(n_3+n_4+n_5+n_7)/2\TBstrut$  & $\pi_{\varphi*}L_0/2-(n_1+n_2+n_3+n_4+n_5+n_8)/2$\TBstrut  \\
\hline
\end{tabular}

\begin{tabular}{|m{0.08\textwidth}<{\centering}|m{0.43\textwidth}<{\centering}|}
\hline
$L_{2,0}^{(1)}(h) \TBstrut$ &{ any $h$ \TBstrut }                       \\  \hline
$E_1\Tstrut $ & $\pi_{\tau*}L_{2,0}^{(1)}/2-(n_1+n_4+n_6+n_8)/2\Tstrut$ \\
$E_2\TBstrut $ & $\pi_{\tau*}L_{2,0}^{(1)}/2-(n_2+n_3+n_5+n_7)/2\TBstrut$ \\
 \hline
$F_1\Tstrut $ & $\pi_{\varphi*}L_{2,0}^{(1)}/2\Tstrut$ \\
$F_2\TBstrut $ & $\pi_{\varphi*}L_{2,0}^{(1)}/2-\sum_{i=1}^8 n_i/2\TBstrut$\\ \hline
\end{tabular}\  
\begin{tabular}{|m{0.08\textwidth}<{\centering}|m{0.3\textwidth}<{\centering}|}
\hline
$L^{(2)}_{2,0}(h) \TBstrut$ &{ any $h$ \TBstrut }                       \\  \hline
$E_1\Tstrut $ & $\pi_{\tau*}L^{(2)}_{2,0}/2 \Tstrut$ \\
$E_2\TBstrut $ & $\pi_{\tau*}L^{(2)}_{2,0}/2-\sum_{i=1}^8 n_i/2\TBstrut$ \\
 \hline
$F_1\Tstrut $ & $\pi_{\varphi*}L_{2,0}^{(2)}/2\Tstrut$ \\
$F_2\TBstrut $ & $\pi_{\varphi*}L_{2,0}^{(2)}/2-\sum_{i=1}^8 n_i/2\TBstrut$\\ \hline
\end{tabular}

\begin{tabular}{|m{0.08\textwidth}<{\centering}|m{0.43\textwidth}<{\centering}|}
\hline
$L_{2,2}^{(a)}(h) \TBstrut$ &{ any $h$ \TBstrut }                       \\  \hline
$E_1\Tstrut $ & $\pi_{\tau*}L_{2,2}^{(a)}/2-(n_1+n_4+n_6+n_8)/2\Tstrut$ \\
$E_2\TBstrut $ & $\pi_{\tau*}L_{2,2}^{(a)}/2-(n_2+n_3+n_5+n_7)/2\TBstrut$ \\
 \hline
$F_1\Tstrut $ & $\pi_{\tau*}L_{2,2}^{(a)}/2-(n_1+n_5+n_6+n_7)/2\Tstrut$ \\
$F_2\TBstrut $ & $\pi_{\tau*}L_{2,2}^{(a)}/2-(n_2+n_3+n_4+n_8)/2\TBstrut$ \\ \hline
\end{tabular}\  \begin{tabular}{|m{0.09
\textwidth}<{\centering}|m{0.29\textwidth}<{\centering}|}
\hline
$L_{4,\pm 4}(h) \TBstrut$ &{ any $h$ \TBstrut }                       \\  \hline
$E_1\Tstrut $ & $\pi_{\tau*}L_{4,\pm 4}/2\Tstrut$ \\
$E_2\TBstrut $ & $\pi_{\tau*}L_{4,\pm 4}/2-\sum_{i=1}^8 n_i/2\TBstrut$ \\
 \hline
$F_1\Tstrut $ & $\pi_{\varphi*}L_{4,\pm 4}/2\Tstrut$ \\
$F_2\TBstrut $ & $\pi_{\varphi*}L_{4,\pm 4}/2-\sum_{i=1}^8 n_i/2\TBstrut$ \\ \hline
\end{tabular}
\\
\begin{center}\begin{tabular}{|m{0.08\textwidth}<{\centering}|m{0.4\textwidth}<{\centering}|}
\hline
$L^{(b)}_{2,2}(h) \TBstrut$ &{ any $h$ \TBstrut }                       \\  \hline
$E_1\Tstrut $ & $\pi_{\tau*}L^{(b)}_{2,2}/2 \Tstrut$ \\
$E_2\TBstrut $ & $\pi_{\tau*}L^{(b)}_{2,2}/2-\sum_{i=1}^8 n_i/2\TBstrut$ \\
 \hline
$F_1\Tstrut $ & $\pi_{\varphi*}L_{2,2}^{(b)}/2-(n_1+n_5+n_6+n_7)/2\Tstrut$ \\
$F_2\TBstrut $ & $\pi_{\varphi*}L_{2,2}^{(b)}/2-(n_2+n_3+n_4+n_8)/2\TBstrut$ \\ \hline
\end{tabular}
\end{center}
\normalsize

\subsection{Eigenspaces  and classes in $NS(\tilde Y)$}\label{sec:eigenspacesKlein}

To determine the action of $(\mathbb Z/2\mathbb Z)^2$, we then have to consider how the residual involutions $\hat\varphi,\hat \tau$ on $\tilde Z_\tau,\tilde Z_\varphi$ act on the divisors $E_i,F_j$ defined in Section \ref{eigenspaces:tau,phi}. In particular, recall from Propositions \ref{pirho}, \ref{pihattau} the action of the residual involutions on the exceptional curves:
\begin{align*}
\hat\varphi=&(n_1,n_8)(n_2,n_5)(n_3,n_7)(n_4,n_6)\\
\hat\tau=&(n_1,n_5)(n_2,n_4)(n_3,n_8)(n_6,n_7).
\end{align*}

\begin{theorem}\label{prop:autospaziKlein}
Let $X$ be a K3 surface that admits a symplectic action of $(\mathbb Z/2\mathbb Z)^2$, and let $L$ be an ample divisor on $X$ invariant for this action. We distinguish two cases: 
\begin{enumerate}
\item Let $L^2=2d=_4 0$, $NS(X)=\Omega_{2,2}\oplus\mathbb ZL$: then the action of $(\mathbb Z/2\mathbb Z)^2$ on $\mathbb P(H^0(X,L)^*)$ is induced by an action of $\mathcal D_4$, the dihedral group of order 8, on $H^0(X,L)$ as follows. 
\[\mathcal D_4=\langle a,b\mid a^2=b^2=1, (ab)^4=1 \rangle\]
\begin{align*}a:& (x_0:\dots :x_{d/2+1}:x_{d/2+2}:\dots : x_{d+2})\mapsto  (x_0:\dots :x_{d/2+1}:-x_{d/2+2}:\dots :-x_{d+2})\\
b:& (x_0:\dots :x_{d/2+1}:x_{d/2+2}:\dots :x_{d+2})\mapsto  (x_{d/2+2}:\dots :x_{d+2}:x_{0}:\dots :x_{d/2+1}).\end{align*}
\item For any other deformation family, there exist divisors $D_1,\dots , D_4\in NS(\tilde Y)$ such that
\[H^0(X,L)=\pi_{2,2}^*H^0(\tilde Y, D_1)\oplus\pi_{2,2}^*H^0(\tilde Y, D_2)\oplus\pi_{2,2}^*H^0(\tilde Y, D_3)\oplus\pi_{2,2}^*H^0(\tilde Y, D_4)\]
and each $\pi_{2,2}^*H^0(\tilde Y, D_i)$ corresponds to one of the subspaces which are the intersection of eigenspaces for the action of the two generators of $(\mathbb Z/2\mathbb Z)^2$ on $H^0(X,L)$:
\[H^0(X,L)=V_{++}\oplus V_{+-} \oplus V_{-+}\oplus V_{--}.\]
\end{enumerate}
\end{theorem}

\begin{proof}
For each projective family, consider for the associated ample class $L$ in Table \ref{esempiLKlein} the divisors $E_1,E_2$ defined in section \ref{eigenspaces:tau,phi}: the residual involution $\hat\varphi$ on $\tilde Z_\tau$ fixes $E_1,E_2$ in all cases, except for $L_0(d)$ and $d$ even, when they are exchanged. The same holds for the action of $\hat\tau$ on $\tilde Z_\varphi$ and the divisors $F_1,F_2$.\\
If $E_i,F_i$ are fixed by the residual involution, we can split $H^0(X,L)$ in four subspaces $V_{+,+}(L),V_{+,-}(L),V_{-,+}(L),V_{-,-}(L)$, each spanned by $\pi_{2,2}^*H^0(\tilde Y, D_i)$ for some nef divisors of the quotient surface: the proof follows the same argument of the cyclic case (see \cite[Prop. 6.1.1]{P}), using the divisors $D_i$ defined in the tables below. In Table \ref{Tab:dimK} the Euler characteristics of the $D_i$ are computed.
\begin{center}
\begin{tabular}{|m{0.06\textwidth}<{\centering}|m{0.48\textwidth}<{\centering}|m{0.4\textwidth}<{\centering}|}
\hline
$L_0(d) \TBstrut$ &{ $d=_4 1$ \TBstrut }                      &{ $d=_4 3$  \TBstrut} \\  \hline
$D_1\Tstrut $ & $\frac{\pi_{2,2*}L_0}{4}-\frac{\overline n_1+m_1+m_3+m_4+m_5}{2}$ \Tstrut & $\frac{\pi_{2,2*}L_0}{4}-\frac{\overline n_1+m_1+m_6}{2}$\Tstrut  \\
$D_2\Tstrut $ & $\frac{\pi_{2,2*}L_0}{4}-\frac{\overline n_1+m_2+m_6+m_7+m_8}{2}\Tstrut$  & $\frac{\pi_{2,2*}L_0}{4}-\frac{\overline n_1+m_2+m_3+m_4+m_5+m_7+m_8}{2}$\Tstrut  \\
$D_3\Tstrut$ & $\frac{\pi_{2,2*}L_0}{4}-\frac{\overline n_2+\overline n_3+\overline n_4+m_2+m_3+m_4+m_5+m_7+m_8}{2}\Tstrut$  & $\frac{\pi_{2,2*}L_0}{4}-\frac{\overline n_2+\overline n_3+\overline n_4+m_2+m_6+m_7+m_8}{2}\Tstrut$\\
$D_4\TBstrut $ & $\frac{\pi_{2,2*}L_0}{4}-\frac{\overline n_2+\overline n_3+\overline n_4+m_1+m_6}{2}$\TBstrut  & $\frac{\pi_{2,2*}L_0}{4}-\frac{\overline n_2+\overline n_3+\overline n_4+m_1+m_3+m_4+m_5}{2}$ \TBstrut\\ \hline
\end{tabular}
\begin{tabular}{|m{0.07\textwidth}<{\centering}|m{0.44\textwidth}<{\centering}|m{0.44\textwidth}<{\centering}|}
\hline
$L_{2,0}^{(1)}(h) \TBstrut$ &{ $h=_2 0$ \TBstrut }  &{ $h=_2 1$ \TBstrut }                    \\  \hline
$D_1\Tstrut $ & $\frac{\pi_{2,2*}L_{2,0}^{(1)}}{4}-\frac{\overline n_1+\overline n_4+m_1+m_7}{2}\Tstrut$ & $\frac{\pi_{2,2*}L_{2,0}^{(1)}}{4}-\frac{\overline n_1+\overline n_4+m_1+m_3+m_4+m_5+m_6+m_7}{2}$   \\
$D_2\Tstrut $ & $\frac{\pi_{2,2*}L_{2,0}^{(1)}}{4}-\frac{\overline n_1+\overline n_4+m_2+m_3+m_4+m_5+m_6+m_8}{2}\Tstrut$ & $\frac{\pi_{2,2*}L_{2,0}^{(1)}}{4}-\frac{\overline n_1+\overline n_4+m_2+m_8}{2}$   \\
$D_3\Tstrut$ & $\frac{\pi_{2,2*}L_{2,0}^{(1)}}{4}-\frac{\overline n_2+\overline n_3+m_2+m_8}{2}\Tstrut$ &$\frac{\pi_{2,2*}L_{2,0}^{(1)}}{4}-\frac{\overline n_2+\overline n_3+m_2+m_3+m_4+m_5+m_6+m_8}{2}\Tstrut$ \\
$D_4 \TBstrut$ & $\frac{\pi_{2,2*}L_{2,0}^{(1)}}{4}-\frac{\overline n_2+\overline n_3+m_1+m_3+m_4+m_5+m_6+m_7}{2}\TBstrut$ & $\frac{\pi_{2,2*}L_{2,0}^{(1)}}{4}-\frac{\overline n_2+\overline n_3+m_1+m_7}{2}$\TBstrut \\ \hline
\end{tabular}
\begin{tabular}{|m{0.08\textwidth}<{\centering}|m{0.35\textwidth}<{\centering}|m{0.25\textwidth}<{\centering}|}
\hline
$L_{2,0}^{(2)}(h) \TBstrut$ &{ $h=_2 0$ \TBstrut }                      &{ $h=_2 1$ \TBstrut }                       \\  \hline
$D_1\Tstrut $ & $\frac{\pi_{2,2*}L_{2,0}^{(2)}}{4}-\frac{m_3+m_4+m_5+m_6}{2}$\Tstrut & $\frac{\pi_{2,2*}L_{2,0}^{(2)}}{4}$\Tstrut  \\
$D_2\Tstrut $ & $\frac{\pi_{2,2*}L_{2,0}^{(2)}}{4}-\frac{m_1+m_2+m_7+m_8}{2}$ & $\frac{\pi_{2,2*}L_{2,0}^{(2)}}{4}-\mu_1$\Tstrut  \\
$D_3\Tstrut$ & $\frac{\pi_{2,2*}L_{2,0}^{(2)}}{4}-\frac{\overline n_1+\overline n_2+\overline n_3+\overline n_4+\sum_i m_i}{2}$  & $\frac{\pi_{2,2*}L_{2,0}^{(2)}}{4}-\mu_2$\Tstrut  \\
$D_4 \TBstrut$ & $\frac{\pi_{2,2*}L_{2,0}^{(2)}}{4}-\frac{\overline n_1+\overline n_2+\overline n_3+\overline n_4}{2}\TBstrut$  & $\frac{\pi_{2,2*}L_{2,0}^{(2)}}{4}-\mu_1-\mu_2$\Tstrut \\ \hline
\end{tabular}
\begin{tabular}{|m{0.08\textwidth}<{\centering}|m{0.434\textwidth}<{\centering}|m{0.434\textwidth}<{\centering}|}
\hline
$L^{(a)}_{2,2}(h) \TBstrut$ &{ $h=_2 0$ \TBstrut }                      &{ $h=_2 1$  \TBstrut} \\  \hline
$D_1\Tstrut $ & $\frac{\pi_{2,2*}L^{(a)}_{2,2}}{4}-\frac{\overline n_1+\overline n_4+m_1+m_4+m_5+m_7}{2} \Tstrut$ & $\frac{\pi_{2,2*}L^{(a)}_{2,2}}{4}-\frac{\overline n_1+\overline n_4+m_1+m_3+m_6+m_7}{2}$\Tstrut  \\
$D_2\Tstrut $ & $\frac{\pi_{2,2*}L^{(a)}_{2,2}}{4}-\frac{\overline n_1+\overline n_4+m_2+m_3+m_6+m_8}{2}\Tstrut$  & $\frac{\pi_{2,2*}L^{(a)}_{2,2}}{4}-\frac{\overline n_1+\overline n_4+m_2+m_4+m_5+m_8}{2}$\Tstrut  \\
$D_3\Tstrut$ & $\frac{\pi_{2,2*}L^{(a)}_{2,2}}{4}-\frac{\overline n_2+\overline n_3+m_2+m_4+m_5+m_8}{2}\Tstrut$  & $\frac{\pi_{2,2*}L^{(a)}_{2,2}}{4}-\frac{\overline n_2+\overline n_3+m_2+m_3+m_6+m_8}{2}\Tstrut$\\
$D_4 \TBstrut$ & $\frac{\pi_{2,2*}L^{(a)}_{2,2}}{4}-\frac{\overline n_2+\overline n_3+m_1+m_3+m_6+m_7}{2}\TBstrut$ & $\frac{\pi_{2,2*}L^{(a)}_{2,2}}{4}-\frac{\overline n_2+\overline n_3+m_1+m_4+m_5+m_7}{2}\TBstrut$ \\ \hline
\end{tabular}
\begin{tabular}{|m{0.08\textwidth}<{\centering}|m{0.434\textwidth}<{\centering}|m{0.434\textwidth}<{\centering}|}
\hline
$L^{(b)}_{2,2}(h) \TBstrut$ &{ $h=_2 0$ \TBstrut }                      &{ $h=_2 1$  \TBstrut} \\  \hline
$D_1\Tstrut $ & $\frac{\pi_{2,2*}L^{(b)}_{2,2}}{4}-\frac{m_3+m_6}{2} \Tstrut$ & $\frac{\pi_{2,2*}L^{(b)}_{2,2}}{4}-\frac{m_4+m_5}{2}$\Tstrut  \\
$D_2\Tstrut $ & $\frac{\pi_{2,2*}L^{(b)}_{2,2}}{4}-\frac{m_1+m_2+m_4+m_5+m_7+m_8}{2}\Tstrut$  & $\frac{\pi_{2,2*}L^{(b)}_{2,2}}{4}-\frac{m_1+m_2+m_3+m_6+m_7+m_8}{2}$\Tstrut  \\
$D_3\Tstrut$ & $\frac{\pi_{2,2*}L^{(b)}_{2,2}}{4}-\frac{\sum_j \overline n_j+m_1+m_2+m_3+m_6+m_7+m_8}{2}\Tstrut$  & $\frac{\pi_{2,2*}L^{(b)}_{2,2}}{4}-\frac{\sum_j \overline n_j+m_1+m_2+m_4+m_5+m_7+m_8}{2}\Tstrut$\\
$D_4 \TBstrut$ & $\frac{\pi_{2,2*}L^{(b)}_{2,2}}{4}-\frac{\sum_j \overline n_j+m_4+m_5}{2}\TBstrut$ & $\frac{\pi_{2,2*}L^{(b)}_{2,2}}{4}-\frac{\sum_j \overline n_j+m_3+m_6}{2}\TBstrut$ \\ \hline
\end{tabular}
\begin{tabular}{|m{0.1\textwidth}<{\centering}|m{0.4\textwidth}<{\centering}|}
\hline
$L_{4,\pm 4}(h) \TBstrut$ &{ any $h$ \TBstrut }                       \\  \hline
$D_1\Tstrut $ & $\frac{\pi_{2,2*}L_{\pm 4,4}}{4}$  \\
$D_2\Tstrut $ & $\frac{\pi_{2,2*}L_{\pm 4,4}}{4}-\mu_1\Tstrut$  \\
$D_3\Tstrut$ & $\frac{\pi_{2,2*}L_{\pm 4,4}}{4}-\mu_2\Tstrut$  \\
$D_4 \TBstrut$ & $\frac{\pi_{2,2*}L_{\pm 4,4}}{4}-\mu_1-\mu_2\TBstrut$  \\ \hline
\end{tabular}\end{center}

\begin{center}
\captionof{table}{Euler characteristics\label{Tab:dimK}}
\begin{tabular}{|m{0.1\textwidth}<{\centering}|m{0.03\textwidth}<{\centering}|m{0.07\textwidth}<{\centering}|m{0.15\textwidth}<{\centering}m{0.15\textwidth}<{\centering}m{0.15\textwidth}<{\centering}m{0.15\textwidth}<{\centering}|}
\cline{2-7}
\nocell{1} &no.\TBstrut & $L\TBstrut$ & $\chi(D_1)\TBstrut$  & $\chi(D_2)\TBstrut$ & $\chi(D_3)\TBstrut$ & $\chi(D_4)$\TBstrut \\ 
\hline
$d=_4 1\TBstrut$ &1\Tstrut & $L_0\TBstrut$   & $(d+3)/4\TBstrut$  & $(d+3)/4\TBstrut$ & $(d-1)/4\TBstrut$ & $(d+3)/4$\TBstrut \\
\hline
$d=_4 3\TBstrut$ &2\Tstrut & $L_0\TBstrut$   & $(d+5)/4\TBstrut$  & $(d+1)/4\TBstrut$ & $(d+1)/4\TBstrut$ & $(d+1)/4$\TBstrut \\
\hline
\multirow{2}{*}{$d=_4 2$\Tstrut} 
&3\Tstrut & $L^{(a)}_{2,2}\Tstrut$  & $(d+2)/4\Tstrut$  & $(d+2)/4\Tstrut$ & $(d+2)/4\Tstrut$ & $(d+2)/4$\Tstrut \\
&4\Tstrut & $L^{(b)}_{2,2}\TBstrut$  & $(d+6)/4\TBstrut$  & $(d+2)/4\TBstrut$ & $(d-2)/4\TBstrut$ & $(d+2)/4$\TBstrut \\
\hline
\multirow{2}{*}{\centering{$d=_8 0\Tstrut$}} 
&5\Tstrut & $L^{(1)}_{2,0}\Tstrut$   & $d/4+1\Tstrut$  & $d/4\Tstrut$ & $d/4+1\Tstrut$ & $d/4$\Tstrut \\
&6\Tstrut & $L^{(2)}_{2,0}\TBstrut$   & $d/4+2\TBstrut$  & $d/4\TBstrut$ & $d/4\TBstrut$ & $d/4$\TBstrut \\
\hline
\multirow{3}{1.5cm}{\centering{$d=_8 4$\TTstrut}} 
&7\Tstrut & $L^{(1)}_{2,0}\Tstrut$  & $d/4\Tstrut$  & $d/4+1\Tstrut$ & $d/4\Tstrut$ & $d/4+1$\Tstrut \\
&8\Tstrut & $L^{(2)}_{2,0}\Tstrut$  & $d/4+1\Tstrut$  & $d/4+1\Tstrut$ & $d/4-1\Tstrut$ & $d/4+1$\Tstrut \\
&9\TBstrut & $L_{\pm 4,4}\TBstrut$ & $d/4+2\TBstrut$  & $d/4\TBstrut$ & $d/4\TBstrut$ & $d/4$\TBstrut \\
\hline
\end{tabular}\end{center}

Consider now the projective family with ample class $L=L_0(d)$, $d$ even: then we cannot split $H^0(X,L)$ in four subspaces, but rather we find
\[H^0(X,L)=V_{+}\oplus V_{-}\]
where $V_+, V_-$ are the eigenspaces for one of the generators of $(\mathbb Z/2\mathbb Z)^2$ (say $\tau$), and the other generator (say $\varphi$) acts exchanging the two.\\
On $H^0(X,L)$ we have an action as follows: choose a basis $\{x_0,\dots,x_{d+2}\}$ of $H^0(X,L)$ such that
\begin{align*}
\tau:& (x_0,\dots ,x_{d/2+1},x_{d/2+2},\dots , x_{d+2})\mapsto  \xi_k(x_0,\dots ,x_{d/2+1},-x_{d/2+2},\dots ,-x_{d+2})
\end{align*}
with $\xi_k$ some root of unity, so that $\tau^2$ is the multiplication by $\xi_k^2$; then it holds 
\begin{align*}
\varphi(x_i)=&\xi_m f_i(x_{d/2+2},\dots ,x_{d+2}) \mathrm{\ for \ every\ } i=0,\dots d/2+1,\\
\varphi(x_j)=&\xi_m f_j(x_{0},\dots ,x_{d/2+1}) \mathrm{\ for \ every\ } j=d/2+2,\dots d+2,
\end{align*}
with $\xi_m$ another root of unity and $f_i$ linear such that $\varphi^2$  is the multiplication by $\xi_m^2$; composing them, we get
\begin{align*}
\varphi(\tau(x_i))=&\xi_k\xi_m f_i(x_{d/2+2},\dots ,x_{d+2}) \mathrm{\ for \ every\ } i=0,\dots, d/2+1\\
\varphi(\tau(x_j))=&-\xi_k\xi_m f_j(x_{0},\dots ,x_{d/2+1}) \mathrm{\ for \ every\ } j=d/2+2,\dots, d+2,
\end{align*}
while 
\begin{align*}
\tau(\varphi(x_i))=&-\xi_k\xi_m f_i(x_{d/2+2},\dots ,x_{d+2}) \mathrm{\ for \ every\ } i=0,\dots, d/2+1\\
\tau(\varphi(x_j))=&\xi_k\xi_m f_j(x_{0},\dots ,x_{d/2+1}) \mathrm{\ for \ every\ } j=d/2+2,\dots, d+2,
\end{align*}
so it holds $\tau\varphi=-\varphi\tau$. Therefore $(\tau\varphi)^2$ is the multiplication by $-\xi_k^2\xi_m^2$. Substituting $\tilde\tau=\tau^k$, and $\tilde\varphi=\varphi^m$, we still get $(\tilde\tau\tilde\varphi)^2=-id$, so $\tilde\tau\tilde\varphi$ has order 4 and $\tilde\tau, \tilde\varphi$ span the dihedral group $\mathcal D_4$ (as anticipated in Remark \ref{rem:bigger_action}); by projectivizing, the action of $\mathcal D_4$ loses faithfulness, and we see on $\mathbb P(H^0(X,L)^*)$ an action of $(\mathbb Z/2\mathbb Z)^2$ via the maps described in the statement. 
\end{proof}

\begin{proposition}
In case 2 of Theorem \ref{prop:autospaziKlein} it holds 
\begin{align*}
\pi_\tau^*H^0(\tilde Z_\tau, E_1)=&\pi_{2,2}^*H^0(\tilde Y, D_1)\oplus \pi_{2,2}^*H^0(\tilde Y, D_2),\\
\pi_\tau^*H^0(\tilde Z_\tau, E_2)=&\pi_{2,2}^*H^0(\tilde Y, D_3)\oplus \pi_{2,2}^*H^0(\tilde Y, D_4).\end{align*}
\end{proposition}
\begin{proof}
See the proof of \cite[Prop. 6.2.2]{P}.
\end{proof}
\begin{remark}
To define $D_1,\dots, D_4$ we chose to use the description of $\tilde Y$ as resolution of the quotient $\tilde Z_\tau/\hat\varphi$. The same results can be obtained using $\tilde Z_\varphi/\hat\tau$ instead.
\end{remark}

\subsection{Projective models with $L^2=4$}\label{exK3Klein}

There are three families of K3 surfaces $X$ polarized with an ample class $L$ such that $L^2=4$: for one of them $L=L_0(2)$, so the action of $(\mathbb Z/2\mathbb Z)^2$ is as described in case 1 of Theorem \ref{prop:autospaziKlein}; the other two correspond to \textbf{no. 3}, \textbf{no. 4} of Table \ref{Tab:dimK}, and we can read from there the dimension of the eigenspaces for the action of $(\mathbb Z/2\mathbb Z)^2$. Moreover, from Theorem \ref{correspondence:GammaKlein}, and in particular looking at the degree of $\hat L_\iota$ (the pseudo-ample class on the intermediate quotient surface $\tilde Z_\iota,\ \iota\in\{\tau,\varphi\}$) and of $\overline L$ on $\tilde Y$, we can see the dimension of the projective space in which the quotients are naturally embedded.

To give equations for the general member of each family, we proceed by firstly defining an action of $(\mathbb Z/2\mathbb Z)^2=\langle\tau,\varphi\rangle$ on the correct projective space ($\mathbb P^3$ for $L=L_0(2),L_{2,2}^{(a)}(0)$, $\mathbb P^1\times \mathbb P^1$ if $L=L_{2,2}^{(b)}(0)$) with eigenspaces of the expected dimension; we then find a family of K3 surfaces which are invariant for this action, noting that each family should have dimension ${7=20-(rk(\Omega_{2,2})+1)}$. To check the simplecticity of the action of $(\mathbb Z/2\mathbb Z)^2$, it is sufficient to check that each of the two generators is a symplectic involution, i.e. that each fixes 8 points on $X$.

Let $L=L_0(2)$: consider the action of $(\mathbb Z/2\mathbb Z)^2$ on $\mathbb P^3$ given by
\begin{align*}
(z_0:z_1:z_2:z_3)&\xmapsto{\tau}  (-z_0:-z_1:z_2:z_3)\\
&\xmapsto{\varphi} (z_2:z_3:z_0:z_1)
\end{align*}
then $\varphi$ exchanges the eigenspaces of $\tau$, which is the action described in Section \ref{eigenspaces:tau,phi} for $L_0(2)$. Quartic surfaces invariant for this action are of the form
\begin{align*}Q_3:\ & q(z_0,z_1)+q(z_2,z_3)+\alpha z_0^2z_2^2+\beta z_0z_1z_2z_3+\gamma z_1^2z_3^2+\delta(z_0^2z_2z_3+z_0z_1z_2^2)+\\&+\varepsilon(z_0^2z_3^2+z_1^2z_2^2)+\zeta(z_0z_1z_3^2+z_1^2z_2z_3)=0;\end{align*}
they depend on 11 parameters, but taking into account projectivities of the form $(z_0:z_1:z_2:z_3)\mapsto(az_0+bz_1:cz_0+dz_1:az_2+bz_3:cz_2+dz_3)$ which commute with the given action of $(\mathbb Z/2\mathbb Z)^2$ we find a moduli space of dimension 7. This is therefore a complete family of K3 surfaces with a symplectic action of $(\mathbb Z/2\mathbb Z)^2$. The quotient surfaces $Z_\tau,Z_\varphi$ admit projective models as complete intersection of 3 quadrics in $\mathbb P^5$, as in \cite[\S 3.4]{VGS}. Since $\overline L^2=16$, we expect $\tilde Y\subset P^9$, so it doesn't admit a natural model as complete intersection of hypersurfaces.

\textbf{no. 3}: Consider the action of $(\mathbb Z/2\mathbb Z)^2$ on $\mathbb P^3$ given by
\begin{align*}
(x_0:x_1:x_2:x_3)&\xmapsto{\tau}  (-x_0:-x_1:x_2:x_3)\\
&\xmapsto{\varphi} (-x_0:x_1:-x_2:x_3)
\end{align*}
then the eigenspaces are all of the same dimension. The family of quartic surfaces
\[Q_4: \sum_{i=0}^3 a_ix_i^4+\sum_{\substack{i,j=0\dots 3 \\ j>i}} b_{ij} x_i^2x_j^2+x_0x_1x_2x_3,\]
whose general member is smooth, is invariant for the action above, and it depends on 7 projective parameters up to the action of projectivities that commute with $\tau,\varphi$.

Since the action of $\tau,\varphi$ is the same up to a change of coordinates, the quotient surfaces $Z_\tau,Z_\varphi$ will be described by similar equations. As in \cite[\S 3.4]{VGS}, we consider the map given by the degree 2 invariants under the action of $\tau$
\[(x_0:x_1:x_2:x_3)\mapsto(x_0^2:x_1^2:x_2^2:x_3^2:x_0x_1:x_2x_3)=(z_0:z_1:z_2:z_3:z_4:z_5);\]
then the surface $Q_4$ maps to the complete intersection of quadrics in $\mathbb P^5$
\begin{equation*} R_4:
    \begin{cases}
       z_4^2=z_0z_1\\
       z_5^2=z_1z_2\\
	z_4z_5=-\sum_{i=0}^3 a_iz_i^2-\sum_{\substack{i,j=0\dots 3 \\ j>i}} b_{ij} z_iz_j
    \end{cases}
\end{equation*}
which is a projective model for $Z_\tau$. 
Now, the automorphism $\hat\rho$ on $\mathbb P^5$ is
\[\hat\rho:(z_0:z_1:z_2:z_3:z_4:z_5)\mapsto(z_0:z_1:z_2:z_3:-z_4:-z_5):\]
the surface $R_4$ has the same form as in \cite[\S 3.7]{VGS}, so its quotient under the action of $\hat\rho$, which is a projective model for $Y$, is the quartic surface in $\mathbb P^3=(z_0:z_1:z_2:z_3)$
\[S_4: z_0z_1z_2z_3+(\sum_{i=0}^3 a_iz_i^2+\sum_{\substack{i,j=0\dots 3 \\ j>i}} b_{ij} z_iz_j)^2=0.\]

\textbf{no. 4}: we have $L^{(b)}_{2,2}(0)=H_1+H_2$ with
\begin{gather*}H_1=\frac{L_0(0)+v_2+w}{2}, \ H_2=\frac{L_0(0)+v_2-w}{2};\quad\langle H_1, H_2\rangle=\begin{bmatrix}0 & 2\\ 2 & 0\end{bmatrix},\\
\tau^*(H_1)=H_2, \quad \varphi^*(H_1)=H_1, \quad \varphi^*(H_2)=H_2.\end{gather*}
Hence, by \cite[Thm. 5.2]{SaintDonat}
\[\phi_{|L^{(b)}_{2,2}(0)|}=\phi_{|H_1+H_2|}: X\xrightarrow{2:1} \mathbb P^1\times\mathbb P^1\]
is a double cover  ramified along a curve $\mathcal B$ of bidegree $(4,4)$ invariant for the action of $(\mathbb Z/2\mathbb Z)^2$ on $\mathbb P^1\times \mathbb P^1$ given by
\begin{align*}(x_0:x_1)(y_0:y_1)&\xmapsto{\tau}(y_0:y_1)(x_0:x_1)\\&\xmapsto{\varphi}(x_0:-x_1)(y_0:-y_1);\end{align*}
curves of this type depend on 7 projective parameters when taking into account the action of the group of projectivities of the form $(x_0:x_1)(y_0:y_1)\mapsto (x_0:ax_1)(y_0:ay_1)$, which are the only ones that commute with the action above. We take the quotient of $X$ by the action of $\tau$ as described in \cite[\S 3.5]{VGS}: the surface $Z_\tau$ is a double cover of $\mathbb P^2=(x_0y_0:x_0y_1+x_1y_0:x_1y_1)=(w_0:w_1:w_2)$ ramified along a sextic curve $\mathcal C$, the union of the image $\mathcal B_\tau$ of $\mathcal B$, which is a quartic curve, and the conic curve invariant for the action of $\hat\varphi$ induced on $\mathbb P^2$,
\[\hat\varphi:(w_0:w_1:w_2)\mapsto(w_0:-w_1:w_2).\]
To find a projective model of $Y$, we map $Z_\tau$ to the space of invariants of degree two of $\hat\varphi$, $\mathbb P^3=(w_0^2:w_1^2:w_2^2:w_0w_2)=(z_0:z_1:z_2:z_3)$: then $Y$ is a double cover of the surface $z_0z_2=z_3^2$ ramified along the cubic curve $\overline{\mathcal C}$ (the image of the sextic curve $\mathcal C$).

Now, let's go back and describe $Z_\varphi$: the action of $\varphi$ on $\mathbb P^1\times\mathbb P^1$ fixes 4 points, which do not belong to the branch curve: therefore, if we write $X: t^2=b$, where $b$ is the polynomial of bidegree (4,4) such that $\mathcal B:b=0$, to have 8 fixed points on $X$ we find that $\varphi$ acts as the identity on $t$. Proceeding as in case no. 3 in \cite[\S 6.3]{P} we embed $\mathbb P^1\times \mathbb P^1$ in $\mathbb P^3$ via the Segre map
\[(x_0:x_1)(y_0:y_1)\mapsto(x_0y_0:x_0y_1:x_1y_0:x_1y_1)=(z_0:z_1:z_2:z_3):\] 
now $X$ is a double cover of $z_0z_3=z_1z_2$, ramified along the image of $\mathcal B$.\\ We consider the induced action of $\varphi$ on the weighted projective space $\mathbb P(2,1,1,1,1)$,
\[\varphi: (t; z_0:z_1:z_2:z_3)\mapsto  (t; z_0:-z_1:-z_2:z_3);\]
the space of invariants of degree 2 for $\varphi$ is $\mathbb P^6=(t:z_0^2:z_1^2:z_2^2:z_3^2:z_0z_3:z_1z_2)=(t:a_0:a_1:a_2:a_3:a_4:a_5)$, and the quotient surface is described by
\begin{equation*}
    \begin{cases}
       a_4=a_5\\
	a_0a_3=a_4^2\\
	a_1a_2=a_5^2\\
	t^2=\overline b
    \end{cases}
\end{equation*}
where $\overline b$ is now a quadric: this is therefore a projective model of $Z_\varphi$ as the complete intersection of 3 quadrics in $\mathbb P^5$, as we expected since $(\pi_{\varphi*}L_{2,2}^{(b)}(0))^2=8$.\\
The action of $\hat\tau$ on $\mathbb P^5$ changes sign to $t$ and exchanges $a_1$ with $a_2$, fixing the other coordinates. Let 
\[\mathbb P^5=(t:c_0:c_1:c_2:c_3:c_4)=(t:a_0:a_1+a_2:a_1-a_2:a_3:a_4):\]
similarly to the surface $S_4$ of case no. 4 in \cite[\S 6.3]{P}, to compute the quotient surface we project from the line $\ell=(\lambda:0:0:\mu:0:0)$ on the invariant space for $\hat\tau$: 
\[\pi:\mathbb P^5\rightarrow\mathbb P^3=(c_0:c_1:c_3:c_4).\]
Then, $Z_\varphi$ covers 4:1 the surface ${c_0^2=c_3c_4}$, and $\hat\tau$ exchanges pairwise the points on each regular fiber: therefore we get again a model of $Y$ as double cover of a quadric surface in $\mathbb P^3$, as expected.


\begin{thebibliography}{10}

\bibitem{GP}
A. Garbagnati, Y. Prieto.
\textit{Order 3 symplectic automorphisms on K3 surfaces}.
Math. Z. \textbf{301} (2022), 225–253.

\bibitem{GS}
A. Garbagnati, A. Sarti.
\textit{Projective models of K3 surfaces with an even set}.
Adv. Geom. \textbf{8} (2008), 413–440.

\bibitem{GS1}
A. Garbagnati, A. Sarti. 
\textit{Elliptic fibrations and symplectic automorphisms on K3 surfaces}.
Comm. Algebra \textbf{37} (2009), 3601–3631.

\bibitem{VGS}
B. van Geemen, A. Sarti.
\textit{Nikulin involutions on K3 surfaces}.
Math. Z. \textbf{255} (2007), 731–753.

\bibitem{Nikulin2}
V. V. Nikulin.
\textit{Finite groups of automorphisms of K\"ahlerian K3 surfaces}.
Russian: Trudy Moskov. Mat. Obshch. \textbf{38} (1979), 75–137. 
English translation: Trans. Moscow Math. Soc.  \textbf{38} (1980), 71–135.

\bibitem{Nikulin1}
V. V. Nikulin.
\textit{Integral symmetric bilinear forms and some of their applications}.
Russian: Izv. Akad. Nauk SSSR Ser. Mat. \textbf{43}, No. 1 (1979), 111–177.
English translation: Math. USSR Izv. \textbf{14}, No. 1 (1980), 103–167.

\bibitem{Nishiyama}
K. Nishiyama.
\textit{The Jacobian fibrations on some K3 surfaces and their Mordell-Weil groups}.
Japan. J. Math. \textbf{22}, No. 2 (1996), 293–347.

\bibitem{SaintDonat}
B. Saint-Donat.
\textit{Projective models of K3 surfaces}. 
Am. J. Math. \textbf{96} (1974), 602–639.

\bibitem{P}
B. Piroddi.
\textit{K3 surfaces with a symplectic automorphism of order 4}.
Math. Nachr. (2024), 1–31.

\bibitem{SchuttShioda}
M. Sch\"utt, T. Shioda.
\textit{Elliptic surfaces}.
Adv. Stud. Pure Math \textbf{60} (2010), 51–160.

\end{thebibliography}
\end{document}